\documentclass[11pt,reqno]{amsart}
\usepackage{amsaddr}  
\usepackage[T1]{fontenc}
\usepackage[notrig]{physics}
\usepackage{tikz}
\usepackage{graphicx,enumerate,amssymb,bbold}
\usepackage[ruled, linesnumbered, noend,vlined]{algorithm2e}
\usepackage[a4paper, left=1.4in, right=1.4in, top=1.6in, bottom=1.3in]{geometry}
\usepackage{float}
\usepackage{inputenc}
\usepackage{bbm}
\usepackage{subcaption}
\usepackage{placeins}
\usepackage{enumitem}
\usepackage{parskip}
\usepackage{setspace}
\usepackage{microtype} 
\usepackage{mathrsfs}

\usepackage[colorlinks=true, pdfstartview=FitV, linkcolor=blue, 
            citecolor=blue, urlcolor=blue]{hyperref}
\usepackage{cleveref}

\numberwithin{equation}{section}
\numberwithin{figure}{section}
\definecolor{mLightBrown}{HTML}{EB811B}

\theoremstyle{plain}
\newtheorem{theorem}{Theorem}[section]
\crefname{theorem}{Theorem}{Theorems}
\newtheorem{lemma}[theorem]{Lemma}
\crefname{lemma}{Lemma}{Lemmata}

\crefname{proposition}{Proposition}{Propositions}
\newtheorem{corollary}[theorem]{Corollary}
\crefname{collary}{Corollary}{Corollaries}
\newtheorem{assumption}[theorem]{Assumption}
\crefname{assumption}{Assumption}{Assumptions}

\newtheorem{remark}{Remark}

\newcommand{\bitem}{\begin{itemize}}
\newcommand{\eitem}{\end{itemize}}
\newcommand{\mc}[1]{\mathcal{#1}}

\newcommand{\R}{\mathbb{R}}

\newcommand{\bpm}{\begin{pmatrix}}
\newcommand{\epm}{\end{pmatrix}}
\newcommand{\bsm}{\left(\begin{smallmatrix}}
\newcommand{\esm}{\end{smallmatrix}\right)}

\DeclareMathOperator{\argmin}{arg min}

\title{
Multilevel Optimization: Geometric Coarse Models and Convergence Analysis\textsuperscript{\ddag}}

\author{ \small{Ferdinand Vanmaele\textsuperscript{\dag}, 
Yara Elshiaty\textsuperscript{*, \dag}, \and Stefania Petra\textsuperscript{\dag}}}

\begin{document}

\begin{abstract}
We study multilevel techniques, commonly used in PDE multigrid literature, 
to solve structured optimization problems. For a given hierarchy of levels, 
we formulate a \emph{coarse model} that approximates the problem at each level 
and provides a descent direction for the fine-grid objective using fewer variables. 
Unlike common algebraic approaches, we assume the objective function and its gradient 
can be evaluated at each level. Under the assumptions of strong convexity and gradient 
$L$-smoothness, we analyze convergence and extend the method to box-constrained optimization. 
Large-scale numerical experiments on a discrete tomography problem 
show that the multilevel approach converges rapidly when far from the solution 
and performs competitively with state-of-the-art methods.

\end{abstract}

\maketitle
\renewcommand{\thefootnote}{\fnsymbol{footnote}}
\footnotetext[3]{This preprint has not undergone peer review or any post-submission improvements 
or corrections. A Version of Record of this contribution will appear in the proceedings of the 
Scale Space and  Variational Methods in Computer Vision (SSVM) 2025 conference, to be published 
in the  Lecture Notes in Computer Science (LNCS) series by Springer. The final published version 
will be available online.
}
\footnotetext[1]{Institute for Mathematics, Heidelberg University \; (\url{elshiaty@math.uni-heidelberg.de})}
\footnotetext[2]{Institute for Mathematics \& Centre for Advanced Analytics and Predictive Sciences (CAAPS), 
University of Augsburg \; (\url{ferdinand-joseph.vanmaele@uni-a.de},\url{stefania.petra@uni-a.de})}

\markboth{\MakeUppercase{Ferdinand Vanmaele, Yara Elshiaty, and Stefania Petra}}{
  \MakeUppercase{Multilevel Optimization: Geometric Coarse Models}}

\section{Introduction}
\label{sec:Introduction}

In this paper, we consider optimization problems of the form  
\begin{equation}  
	\min_{y \in C \subseteq \mathbb{R}^n} f(y), \qquad f \in \mathcal{S}_{\mu,L}^{2,1}(C),  
	\label{eq:f-function-class}  
\end{equation}  
where $\mathcal{S}_{\mu,L}^{2,1}(C)$ represents the class of twice-differentiable, $\mu$-strongly convex functions with $L$-Lipschitz continuous gradients over a closed simple convex set $C$, such as a box. We assume that function $f$ possesses an inherent ``geometry'', meaning that it can be discretized at
multiple levels of resolution, forming a hierarchical  structure.

We employ multilevel optimization, as introduced in \cite{nash_multigrid_2000}, to exploit this hierarchy. The core idea is to reformulate the problem on a coarse grid, referred to as the \emph{coarse model}, where computations involve significantly fewer variables. During optimization on the fine grid, the search direction is determined using the coarse model. This approach not only reduces computational costs but also benefits from improved conditioning of the coarse-grid problems.  
  
The classical model by Nash \cite{nash_multigrid_2000} constructs the coarse representation by linearly modifying the fine objective, evaluated on the coarse grid. The linear term incorporates \emph{gradients} from both the coarse and fine objectives, using appropriate mappings between grid levels. Alternatively, \cite{ho_newton-type_2019} proposes a quadratic coarse model to approximate the fine objective by transferring its \emph{Hessian} to the coarse grid via grid transfer operators.

Under the smoothness and convexity assumptions specified in \eqref{eq:f-function-class}, \cite{ho_newton-type_2019} demonstrates a composite convergence rate for their quadratic coarse model by combining linear and quadratic rates. In this work, we establish the relationship between the two coarse models, showing that both use first- and second-order information from the fine objective, which enables us to derive a similar convergence rate. Crucially, the Nash model \cite{nash_multigrid_2000} avoids the explicit computation of the fine objective's Hessian on either the fine or coarse grid. This makes it particularly advantageous for large-scale problems, such as those encountered in imaging applications.

\vspace{1mm}
\noindent
\textbf{Related work.}
Seminal contributions to unconstrained smooth multilevel optimization include \cite{nash_multigrid_2000,gratton_recursive_2008-1,wen_line_2010}. These foundational ideas were later extended to address box-constrained and simplex-constrained optimization in \cite{mohammed_multigrid_2016,gratton_recursive_2008,gratton_numerical_2008} and further developed for more general nonsmooth convex optimization in \cite{Parpas:2017,IML_Fista,Ang_MG_Nonsmooth}. Convergence results for the unconstrained case, under assumption (\ref{eq:f-function-class}), were established in \cite{ho_newton-type_2019} for a quadratic coarse model solely dependent on the objective $f=f_{0}$. In this work, we extend these results to the general coarse model proposed in \cite{nash_multigrid_2000}.

\vspace{1mm}
\noindent
\textbf{Contribution and organization.}
Section \ref{sec:Multilevel-Optimization} introduces key concepts of multilevel optimization and establishes a connection between the coarse model of Nash \cite{nash_multigrid_2000} and the algebraic coarse model from \cite{ho_newton-type_2019}. In Section \ref{sec:Rate-of-Convergence}, we extend the convergence analysis and prove linear rate for multilevel optimization in the unconstrained setting, generalizing the results of \cite{ho_newton-type_2019} to the Nash model. Section \ref{sec:Tomography} presents large-scale experiments, demonstrating that the multilevel approach performs on par with L-BFGS-B \cite{liu_limited_1989,byrd_limited_1995,zhu_algorithm_1997} and offers significant potential for further refinement.

\vspace{1mm}
\noindent
\textbf{Basic notation.}
$\left\langle \cdot,\cdot\right\rangle $ denotes the standard inner
product on $\mathbb{R}^{n}$, $\nabla f$ the gradient and $\nabla^{2}f$
the Hessian of a sufficiently differentiable function $f:\mathbb{R}^{n}\rightarrow\mathbb{R}$.
We denote componentwise multiplication of vectors by $uv=(u_{1}v_{1},\dots,u_{n}v_{n})^{\top}$
and, for strictly positive vectors $v\in\mathbb{R}_{++}^{n}$, componentwise
division by $\frac{u}{v}$. Like-wise, the functions $e^{x}$ and
$\log x$ apply componentwise to a vector $x$.

\section{Multilevel Optimization}
\label{sec:Multilevel-Optimization}
\textbf{Two level optimization.}
We describe a two-grid cycle for updating $y_{k+1}$ from the current iterate $y_k$. 
This involves either a search direction from a coarse-grid model with fewer variables (\emph{coarse correction}) 
or, when coarse correction is ineffective, a standard local approximation defined on the fine grid 
(\emph{fine correction}). The approach is summarized in Algorithm \ref{alg:Two-level-algorithm}.

We denote the fine objective by $f:{\R}^{n}\to \R$ and $g:{\R}^{N}\to\R$
its discretization on a coarse grid. In this work, we assume that the fine problem can be directly discretized 
on coarser grids, as is common in variational models derived from infinite-dimensional problems, such as tomography.
This \emph{geometric} approach contrasts with \emph{algebraic} methods, where the problem is evaluated solely on 
the fine grid, and coarse-grid representations are constructed using mappings to transfer quantities between grids. 
Specifically, we assume that linear maps $R:\mathbb{R}^{n}\rightarrow\mathbb{R}^{N}$ \emph{(restriction)} and 
$P:\mathbb{R}^{N}\rightarrow\mathbb{R}^{n}$ \emph{(prolongation)} are provided to translate quantities between 
levels, typically via linear interpolation or injection as in classical multigrid methods 
\cite{trottenberg_multigrid_2000}.
We assume the standard Galerkin condition $R=cP^\top$ \cite{ho_newton-type_2019}, where $c>1$ 
(taken 1 for simplicity), and $\rank(P)=n$.

\begin{algorithm}[h]
\DontPrintSemicolon
\textbf{initialization:} Set $k=0$ and choose initial point $y_0$, two grids, transfer operators $R$ and $P$ 
and a coarse representation $g$ of the objective. \;
\Repeat{a stopping rule is met.} {
  \If{condition to use coarse model is satisfied at $y_k$} {
    Define coarse model $\psi_k(x;Ry_k,g,R\nabla f(y_k))$.\tcc*{coarse model}
    Find  $x_k^+$ such that $\psi_k(x_k^+)<\psi_k(x_k)$. \label{mls-find-lesser-value}\;
    Set $d_k=P(x_k^+ -Ry_k)$. \tcc*{descent direction for $f$}
    Find $\alpha_k>0$ such that $f(y_k+\alpha_k d_k) < f(y_k)$. \label{f-lesser-value}\tcc*{line search}
    $y_{k+1}= y_k+ \alpha_k d_k$.
  }\Else{
  Apply one fine-grid iteration to update $y_k$. \label{mls-fine-correction}\;
  }
  Increment $k \leftarrow k+1$.
  }
\caption{Two level optimization \cite{ho_newton-type_2019,wen_line_2010}}\label{alg:Two-level-algorithm}
\end{algorithm}

\begin{figure}[h]
\begin{minipage}[t]{0.2\columnwidth}%
\includegraphics[width=1\columnwidth]{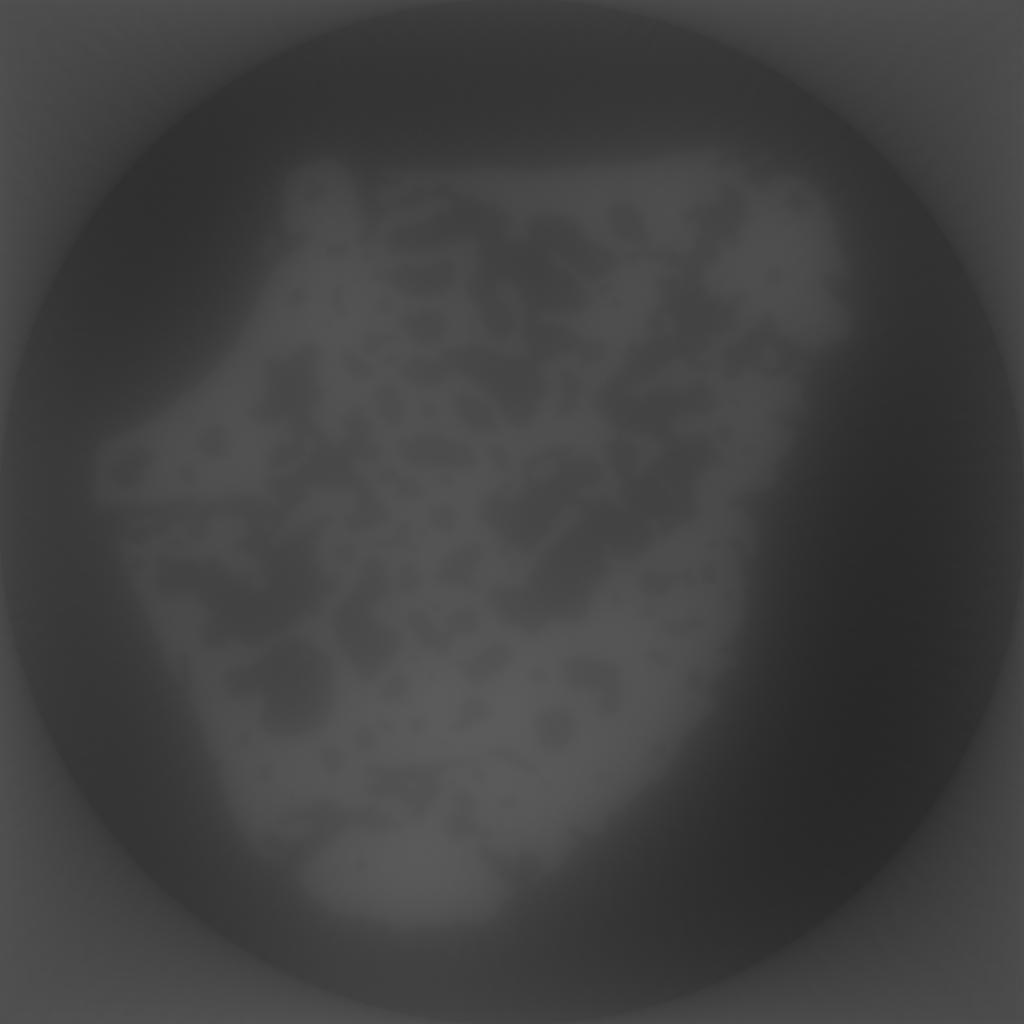}%
\end{minipage}\hfill{}%
\begin{minipage}[t]{0.2\columnwidth}%
\includegraphics[width=1\columnwidth]{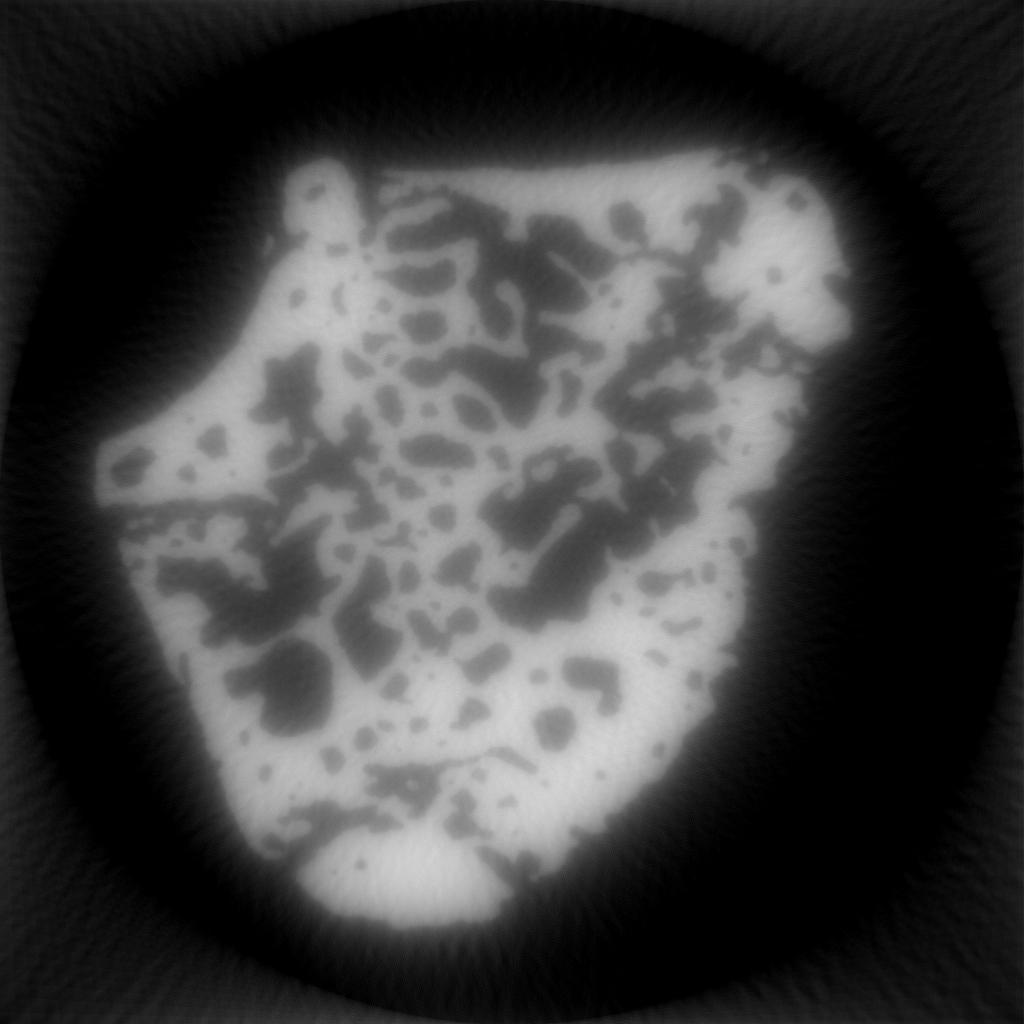}%
\end{minipage}\hfill{}%
\begin{minipage}[t]{0.2\columnwidth}%
\includegraphics[width=1\columnwidth]{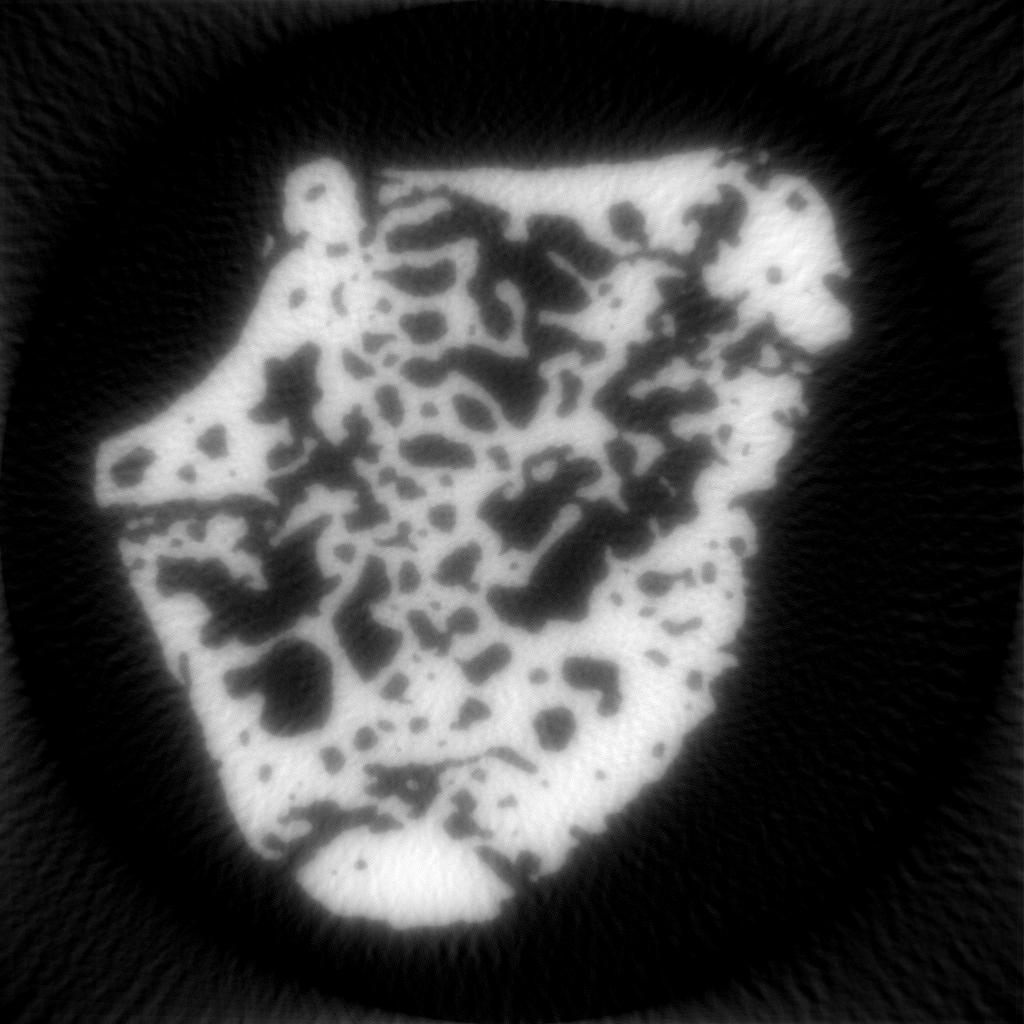}%
\end{minipage}\hfill{}%
\begin{minipage}[t]{0.2\columnwidth}%
\includegraphics[width=1\columnwidth]{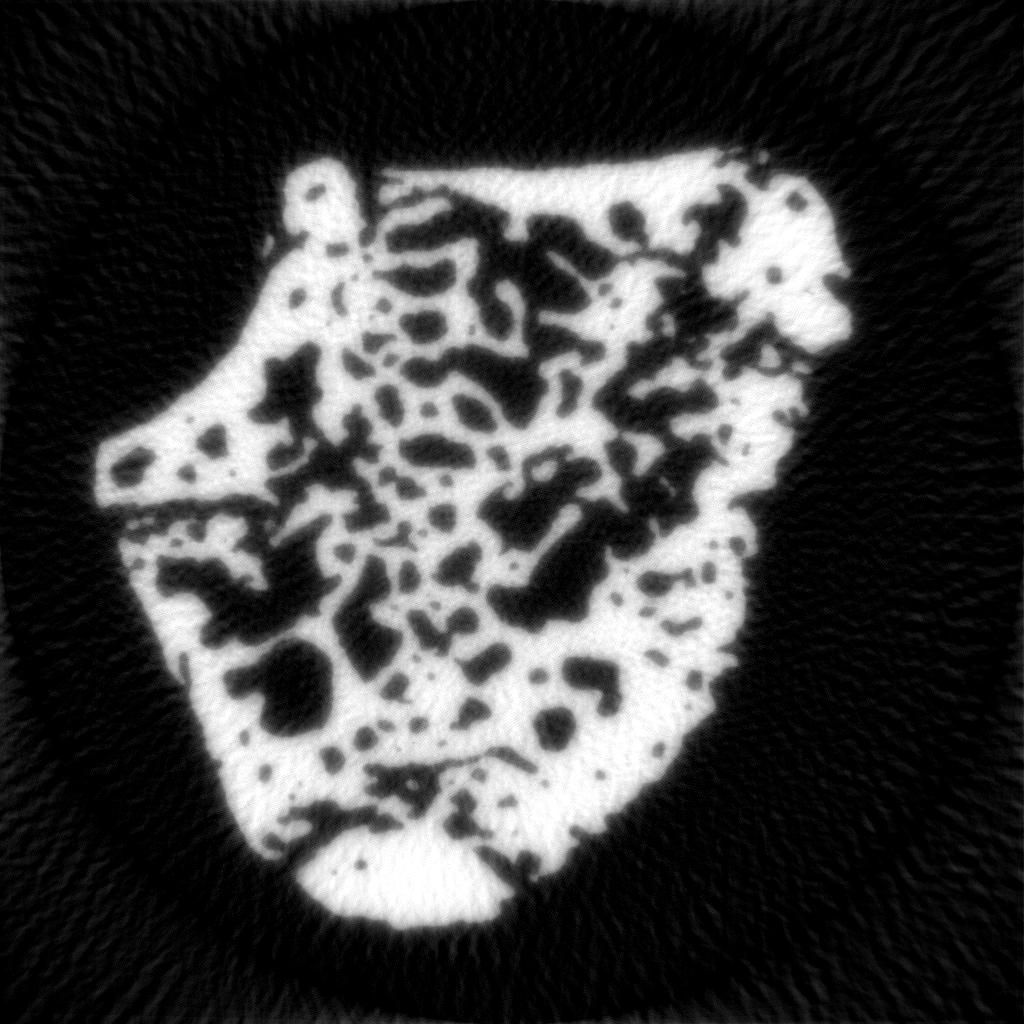}%
\end{minipage}\hfill{}%
\begin{minipage}[t]{0.2\columnwidth}%
\includegraphics[width=1\columnwidth]{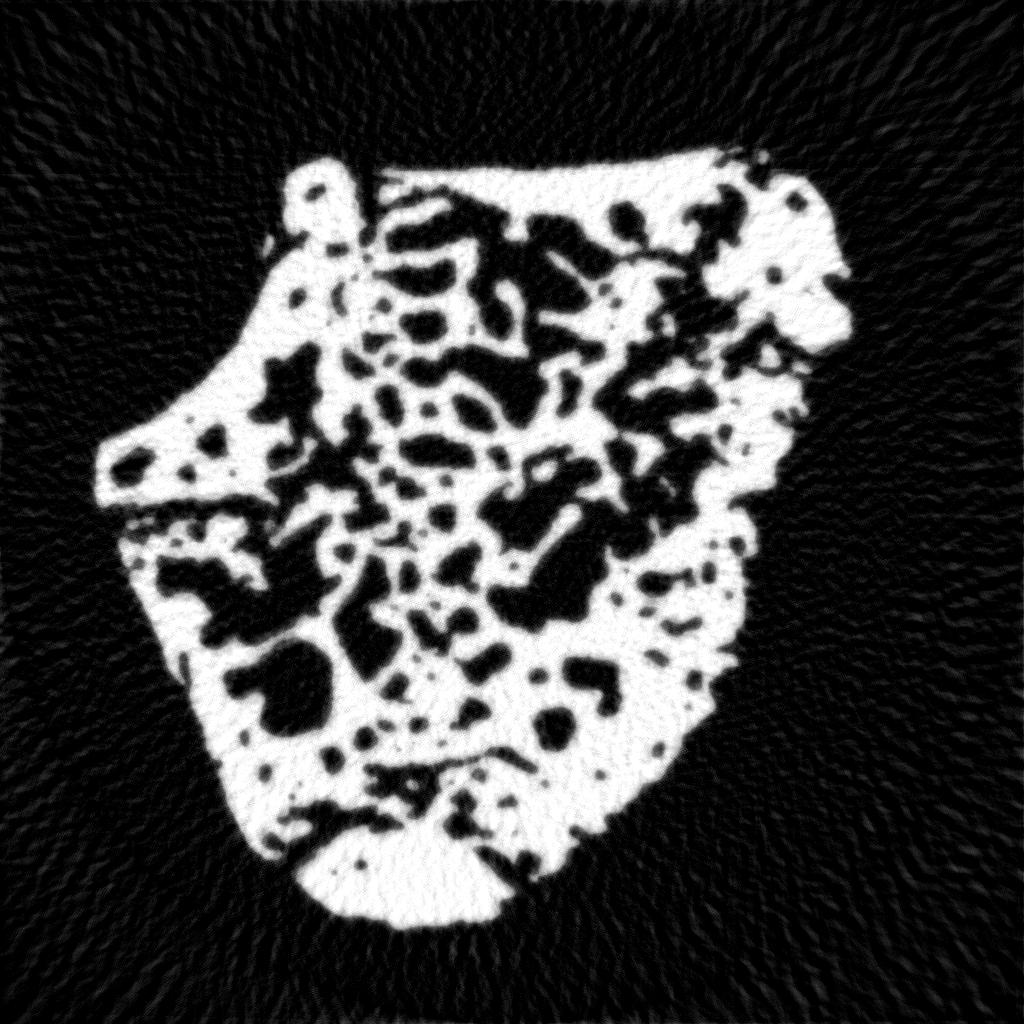}%
\end{minipage}

\begin{minipage}[t]{0.2\columnwidth}%
\includegraphics[width=1\columnwidth]{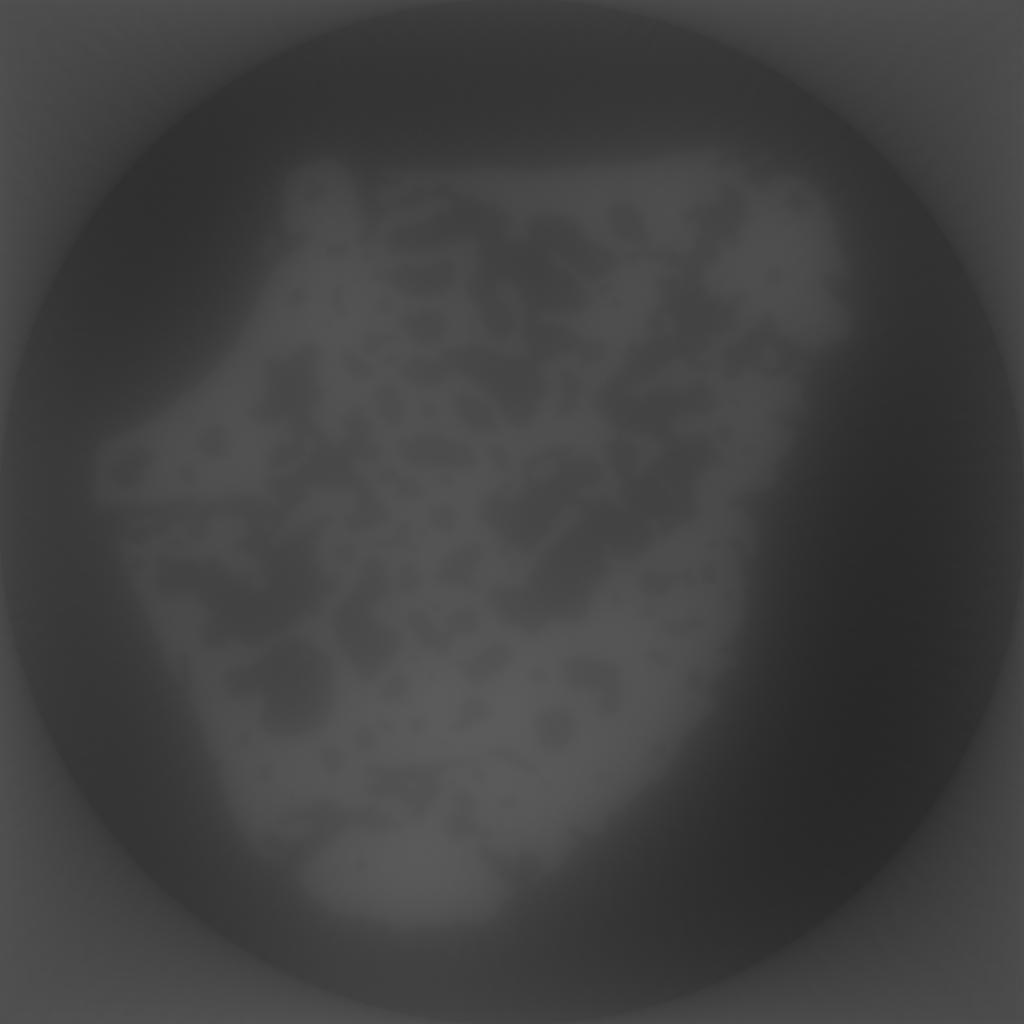}%
\end{minipage}\hfill{}%
\begin{minipage}[t]{0.2\columnwidth}%
\includegraphics[width=1\columnwidth]{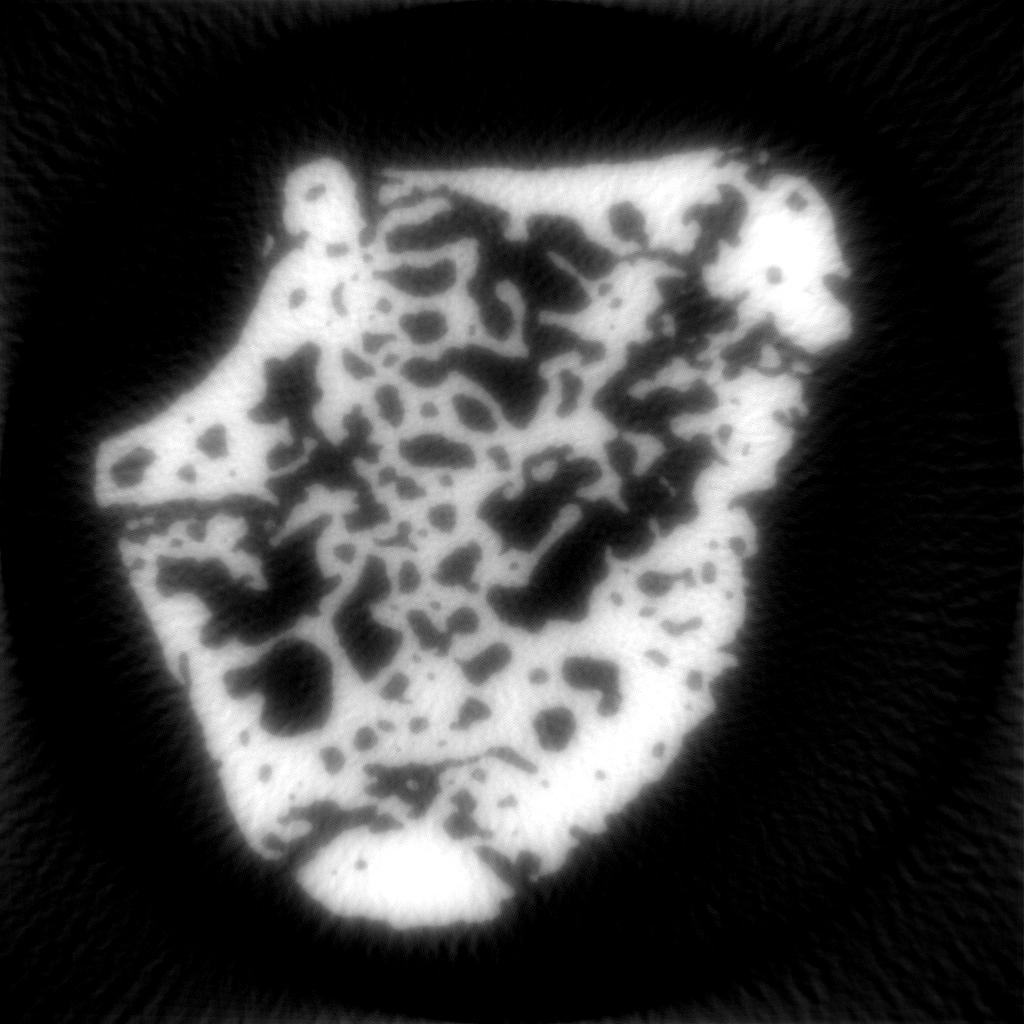}%
\end{minipage}\hfill{}%
\begin{minipage}[t]{0.2\columnwidth}%
\includegraphics[width=1\columnwidth]{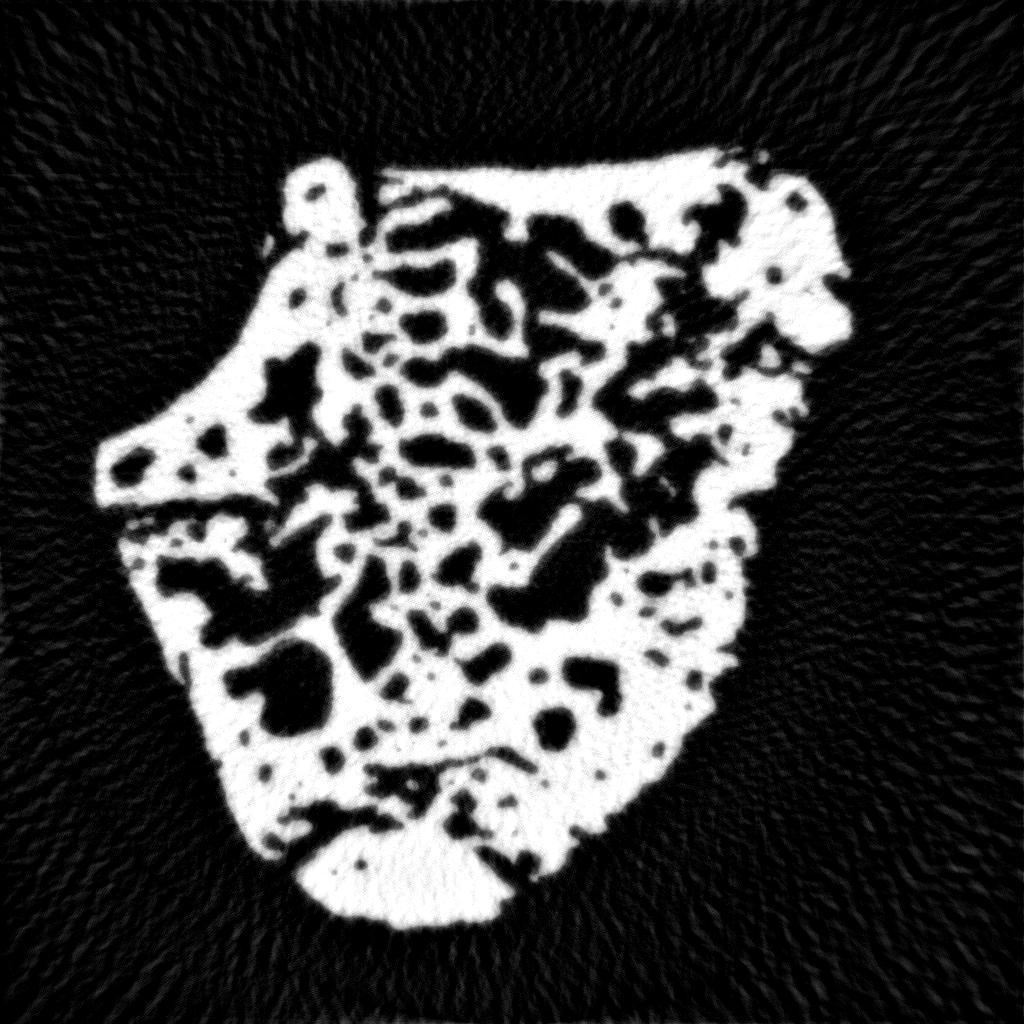}%
\end{minipage}\hfill{}%
\begin{minipage}[t]{0.2\columnwidth}%
\includegraphics[width=1\columnwidth]{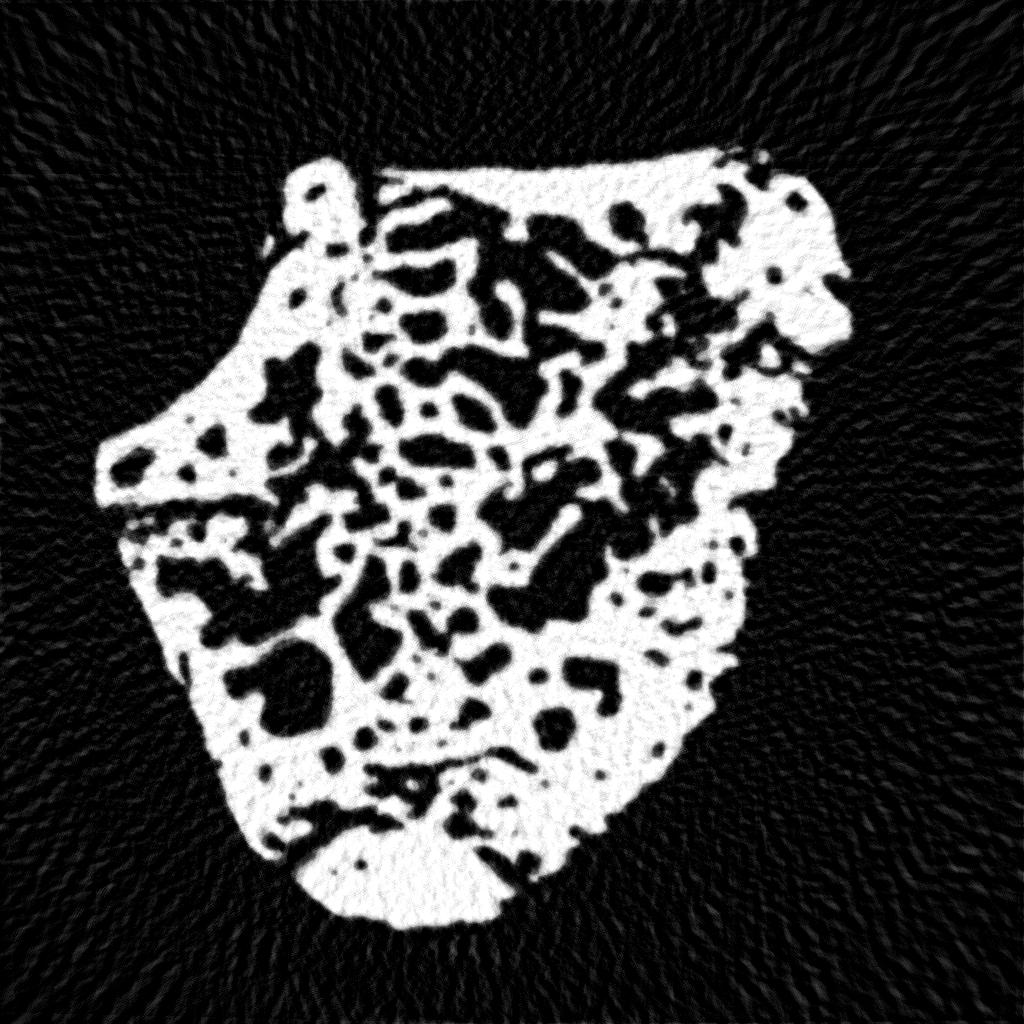}%
\end{minipage}\hfill{}%
\begin{minipage}[t]{0.2\columnwidth}%
\includegraphics[width=1\columnwidth]{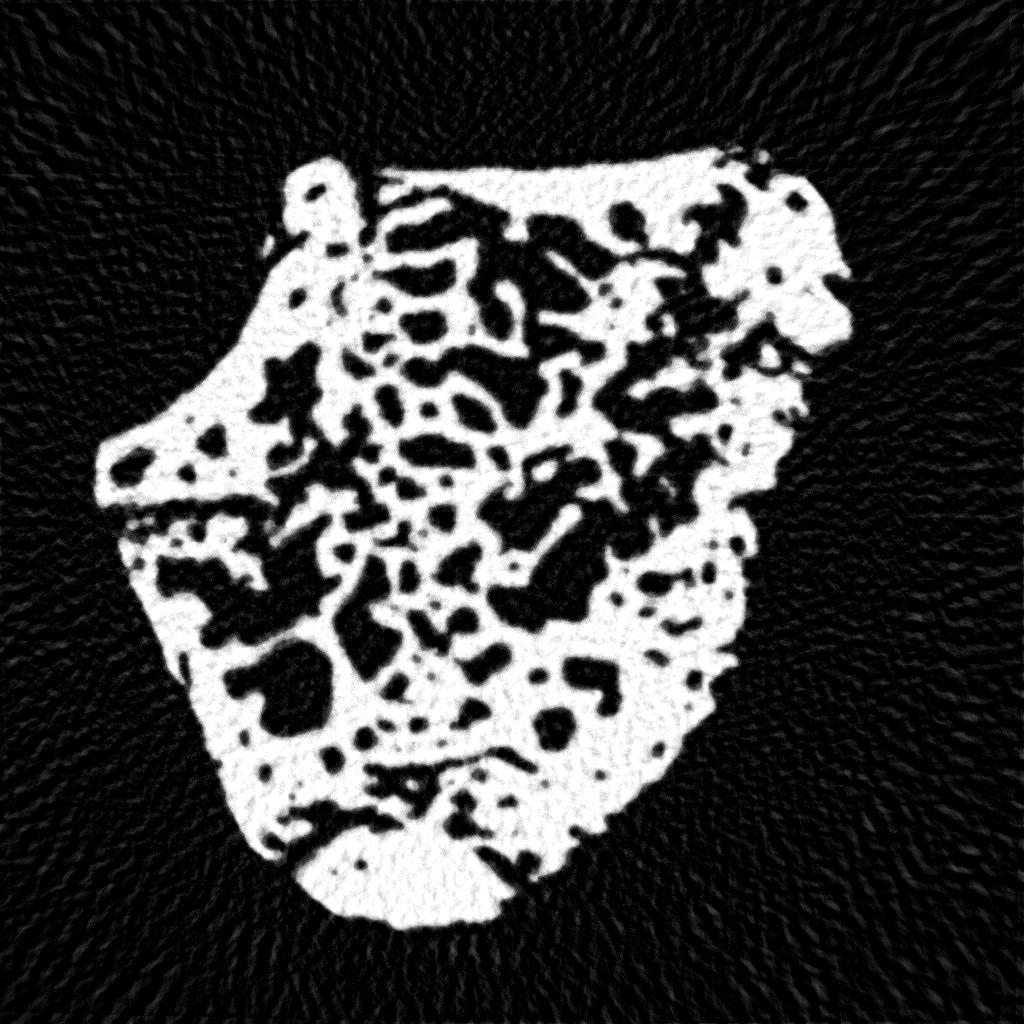}%
\end{minipage}%

\begin{minipage}[t]{0.2\columnwidth}%
\includegraphics[width=1\columnwidth]{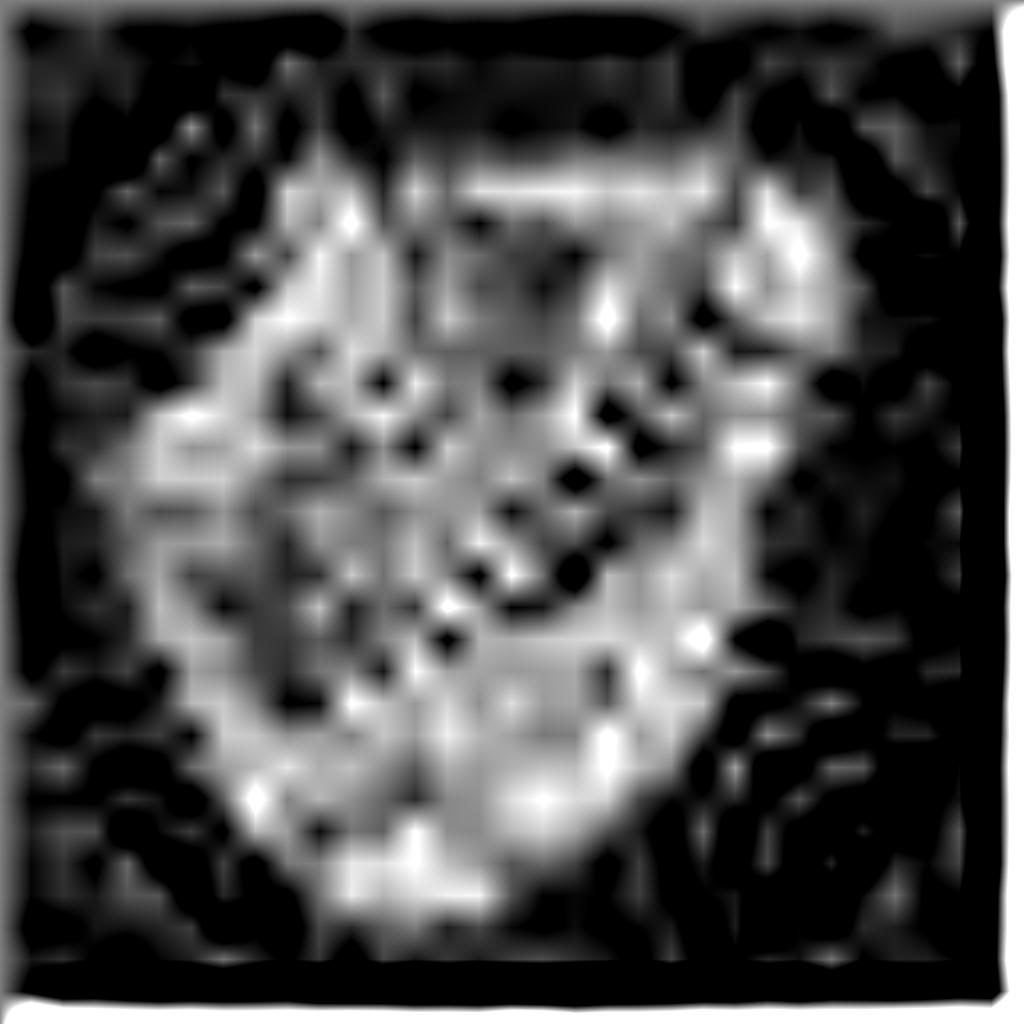}%
\end{minipage}\hfill{}%
\begin{minipage}[t]{0.2\columnwidth}%
\includegraphics[width=1\columnwidth]{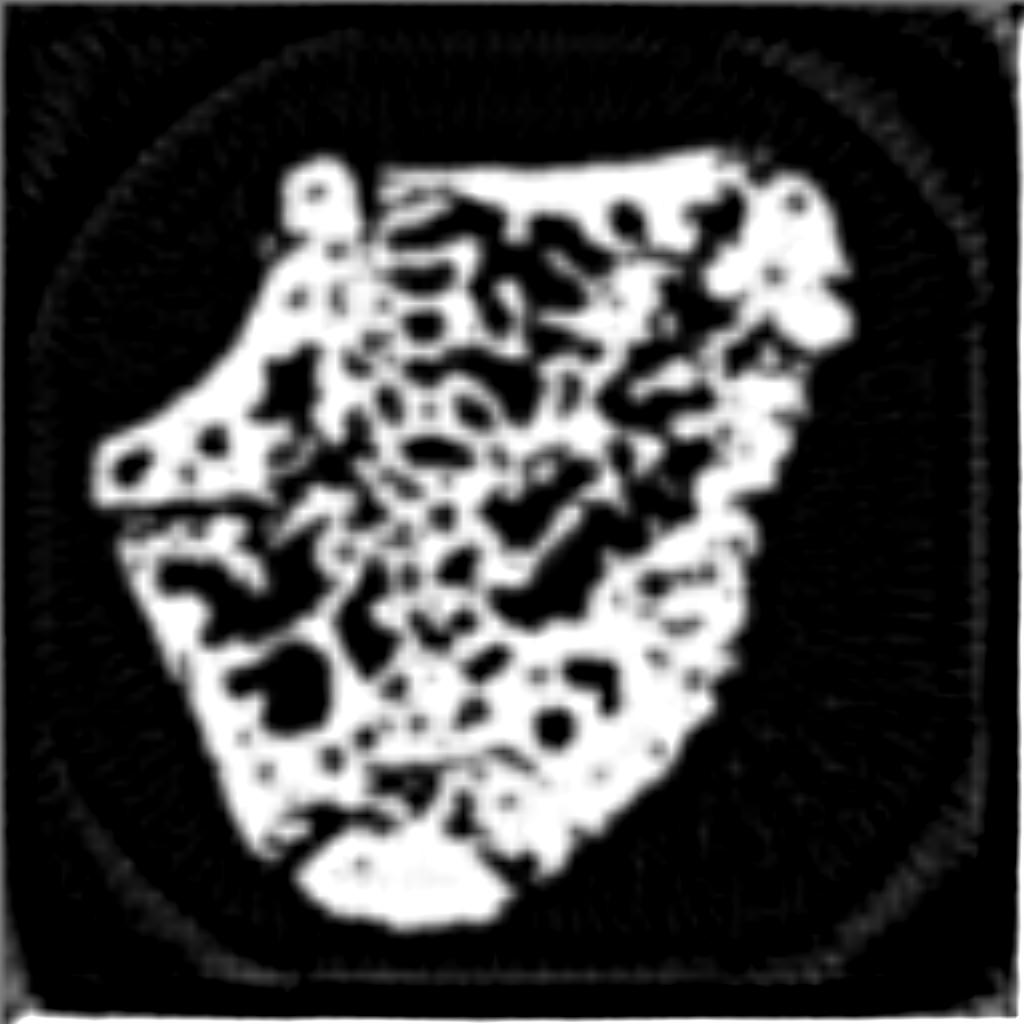}%
\end{minipage}\hfill{}%
\begin{minipage}[t]{0.2\columnwidth}%
\includegraphics[width=1\columnwidth]{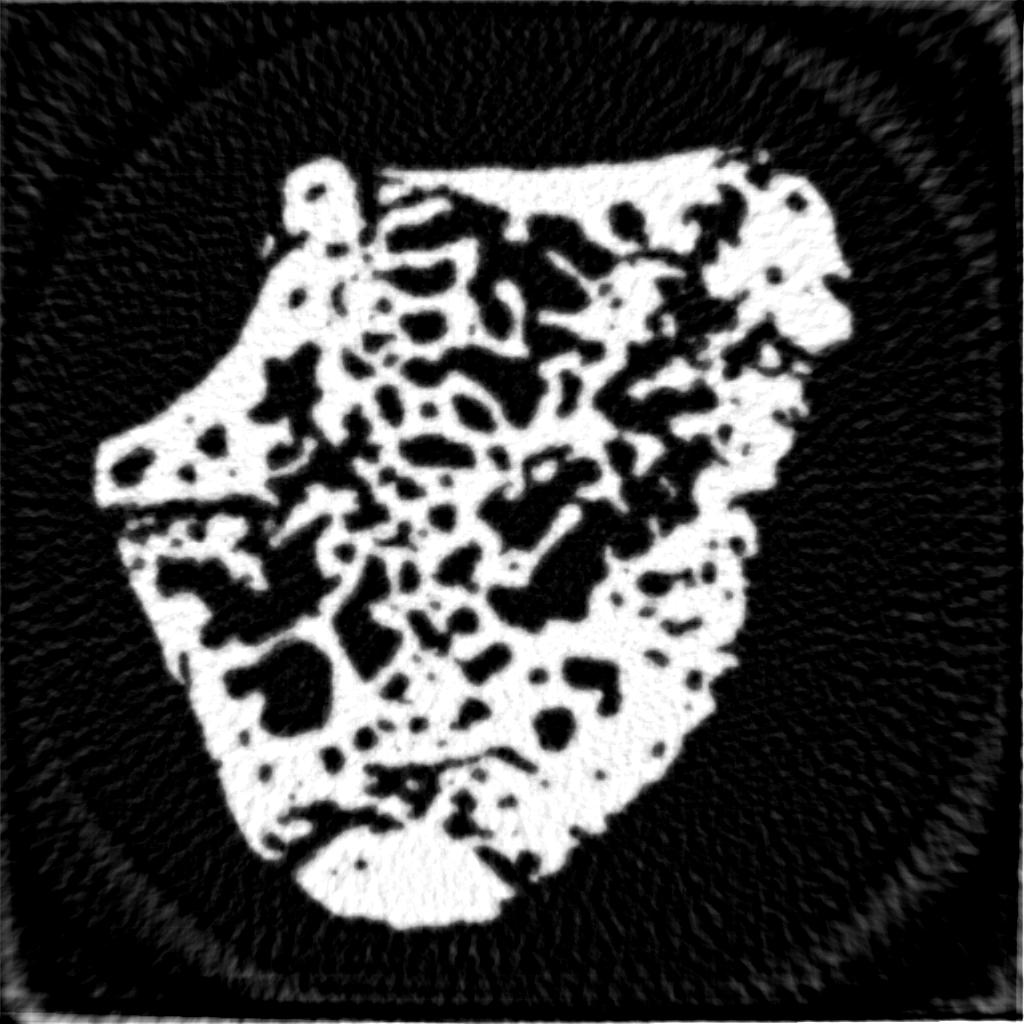}%
\end{minipage}\hfill{}%
\begin{minipage}[t]{0.2\columnwidth}%
\includegraphics[width=1\columnwidth]{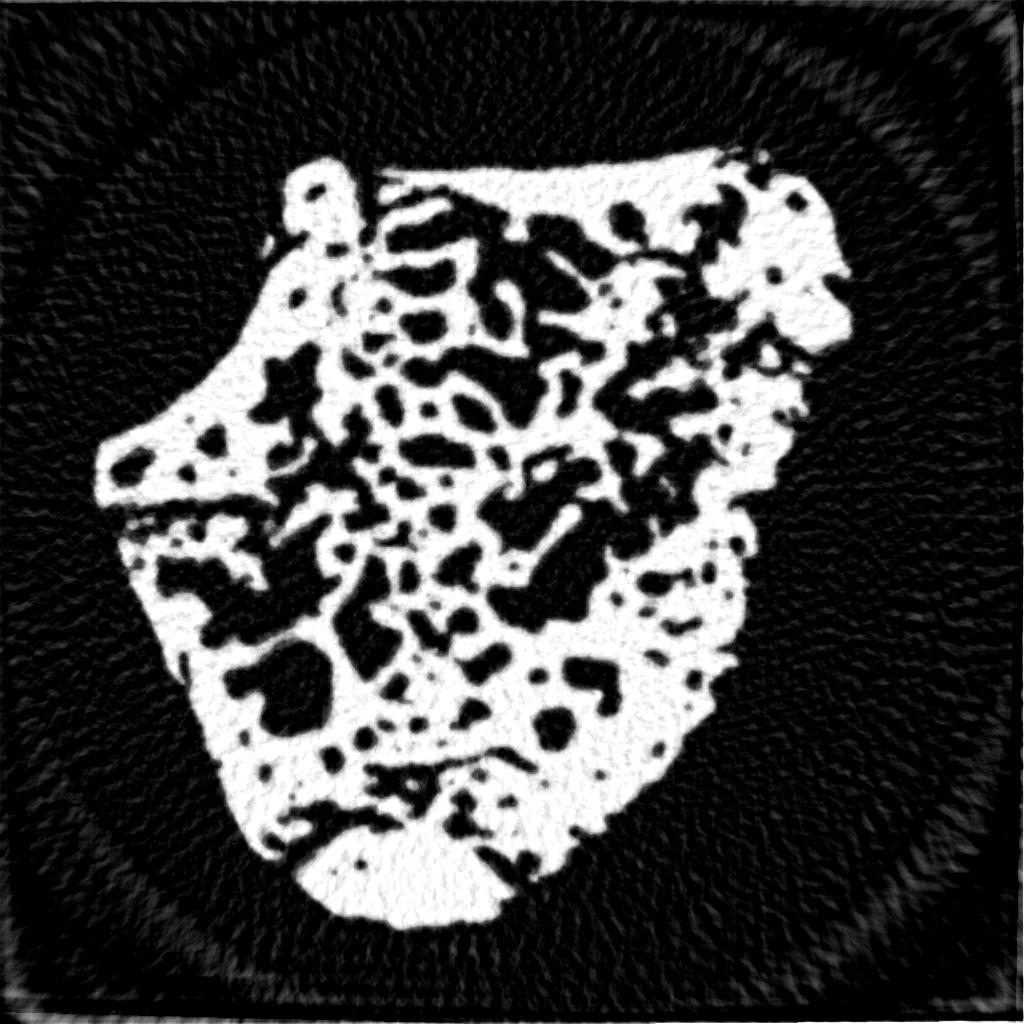}%
\end{minipage}\hfill{}%
\begin{minipage}[t]{0.2\columnwidth}%
\includegraphics[width=1\columnwidth]{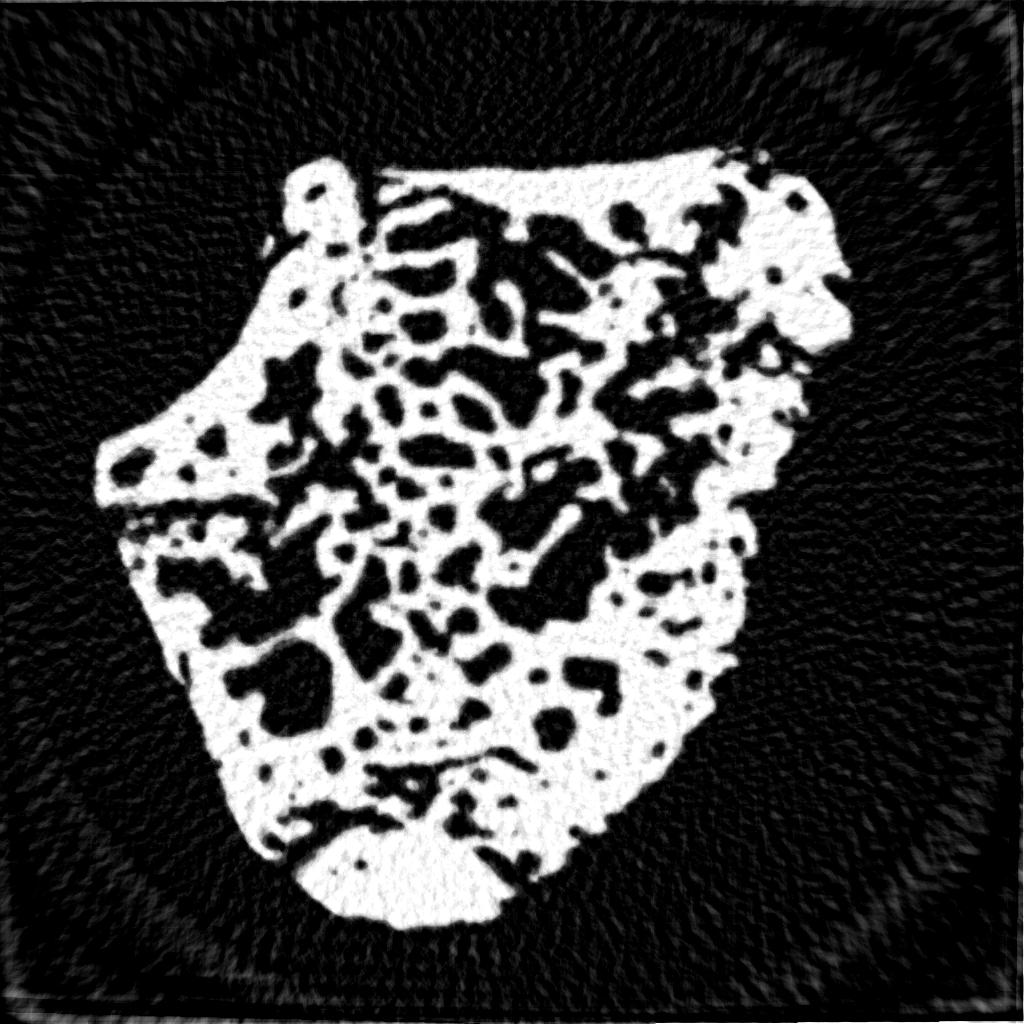}%
\end{minipage}%

\caption{\textbf{Single and multilevel reconstructions.} Comparison of gradient descent (top), L-BFGS (middle), and Algorithm \ref{alg:Two-level-algorithm} (bottom) applied recursively over 6 levels after 
$k=1,5,10,20,50$ iterations for the problem \eqref{eq:tomography-unconstrained}. Multilevel reconstructions yield results closer to the original and outperform BFGS.
}
\label{fig:bone-reconstruction}
\end{figure}

\noindent
\textbf{Coarse model.}
The core idea is to define the \emph{coarse grid model}, introduced in \cite{nash_multigrid_2000},
 analogous to a standard local approximation. It is defined at each step $k$ based on the current iterate as
$y_{k}\in{\R}^{n}$
\begin{align}
\min_{x\in C\subset\mathbb{R}^{N}}\psi_k(x) & :=g(x)+\left\langle v_{k},x-Ry_{k}\right\rangle 
,\quad v_{k}:=R\nabla f(y_{k})-\nabla g(Ry_{k}),
\label{eq:coarse-model-1}
\end{align}
where $\psi_k $ is a linear modification of the coarse objective $ g $. The linear term $ v_k $ 
represents the difference between the restricted gradient $ \nabla f $ evaluated at $ y_k $ 
and the gradient of the coarse objective $ \nabla g $ evaluated at $ Ry_k $. Consequently, the model 
incorporates both algebraic and geometric gradient information.  

For the \emph{initial} coarse-grid iterate $x_{k}:=Ry_{k}$, the gradient of the coarse model 
satisfies 
\begin{equation}
\nabla\psi_k(x_k)=R\nabla f(y_{k}). \label{eq:first-order-coherence}
\end{equation}
This property, known as the \emph{first-order coherence condition}, ensures that a critical point 
on the fine grid remains critical when transferred to the coarse grid. 

The coarse model is designed to efficiently compute a descent direction using fewer coarse-grid variables. 
The fine-grid update is then performed as
\begin{align}
y_{k+1} & =y_{k}+\alpha_{k} d_{k},\quad\alpha_{k}>0,\qquad d_{k}=P(x_{k}^+-Ry_{k}),\label{eq:coarse-update}
\end{align}
where $x_{k}^+$ is an \emph{approximate} solution to the coarse model (\ref{eq:coarse-model-1}) 
satisfying $\psi_k(x_k^+)< \psi_k(x_k)$.
The step size $\alpha_{k}>0$ is determined via a suitable \emph{line search} strategy.

\vspace{1mm}
\noindent
\textbf{Relation to second order models.}
 For $x_{k}:=Ry_{k}$ and $R=P^{\top}$,  the coarse model \eqref{eq:coarse-model-1} can be rewritten as
\begin{equation}\label{eq:coarse-model-Breg} 
\psi_k(x) =g(x_{k})+\left\langle \nabla f(y_{k}),P(x-x_{k})\right\rangle +D_{g}(x,x_{k})
\end{equation}
where the  Bregman divergence is
\begin{align*}
D_{g}(x,x_{g}) & =g(x)-g(x_{k})-\left\langle \nabla g(x_{k}),x-x_{k}\right\rangle 
  =\frac{1}{2}\left\langle x-x_{k},\nabla^{2}g(\tilde{x})(x-x_{k})\right\rangle 
\end{align*}
with $\tilde{x} = (1-t)x+tx_{k}$ for some ${t\in[0,1]}$.
Thus, the coarse model resembles the quadratic approximation used in \emph{single-level} optimization
\begin{equation}\label{eq:local-approx}
q_{k}(y)=\left\langle \nabla f(y_k),y-y_k\right\rangle 
+\frac{1}{2}\left\langle y-y_k ,H_{k} (y-y_k)\right\rangle ,
\end{equation}
where $H_{k}\succ 0$ approximates $\nabla^{2}f(y_k)$.
In contrast, the coarse model $\psi_k$ incorporates first order information
of the fine objective $f$, and second order information of the coarse
objective $g$.



\vspace{1mm}
\noindent
\textbf{Alternative coarse model.}
The authors in \cite{ho_newton-type_2019} define the  coarse model as
\begin{align}
\label{eq:algebraic-coarse}
\varphi_k(x) & = \left\langle R \nabla f(y_{k}),x-x_{k}\right\rangle	+ \frac{1}{2} \langle x - x_k,  
\underbrace{R \nabla^2 f(y_k) P }_{:=Q_k} (x - x_k)\rangle.
\end{align}
The key difference between \eqref{eq:coarse-model-Breg} and \eqref{eq:algebraic-coarse} is that the latter 
explicitly computes the fine-grid Hessian $\nabla ^2 f$ at each iteration, which makes it impractical for
very large problems. However, the update in \eqref{eq:coarse-update}, $y_{k+1}  = y_k + \alpha_k {d}_k $ 
can be computed using a Newton-like direction $ {d}_k=-PQ_k^{-1}R\nabla f(y_{k})$, which corresponds to 
solving \eqref{eq:algebraic-coarse}. This enables the authors to establish (sub)linear convergence and a 
finite number of coarse correction steps. Similar results for the coarse model \eqref{eq:coarse-model-1} 
are presented in Section \ref{sec:Rate-of-Convergence}.

\vspace{1mm}
\noindent
\textbf{Coarse correction descent.} In view of first-order coherence \eqref{eq:first-order-coherence} of the coarse model $\psi_k$, we obtain the following result.
\begin{lemma}
\label{lem:coarse-descent-direction}
If $g$ is a convex function, then the update in \eqref{eq:coarse-update} defines a descent direction for $f$ at $y_k$.  
Moreover, if in the $k$-th iteration, the search direction is given by $d_k = P(x_\ast - x_k)$, where $x_\ast$ is a solution of the coarse model \eqref{eq:coarse-model-Breg}, then  
\begin{equation*}
  \left\langle \nabla f(y_k),d_k\right\rangle \leq\psi_k(x_\ast)-\psi_k(Ry_k)\leq 0.
\end{equation*}
\end{lemma}

\begin{proof} By \eqref{eq:coarse-model-Breg} and 
$\psi_k(x_k) =g(x_{k})$ we have for arbitrary $x$
\begin{align*}
\psi_k(x) & =\psi_k(x_{k})+\left\langle \nabla f(y_{k}),P(x-x_{k})\right\rangle +\underbrace{D_{g}(x,x_{k})}_{\ge 0}\\ &\ge \psi_k(x_{k})+\left\langle \nabla f(y_{k}),P(x-x_{k})\right\rangle,
\end{align*}
 where the last inequality follows by convexity of $g$.
\end{proof}

\vspace{1mm}
\noindent
\textbf{Coarse correction condition.}
We adapt the criteria from \cite{wen_line_2010} for \emph{unconstrained} optimization
\begin{align*}
\|P^{\top}\nabla f(y)\|_{\ast} & \geq\kappa\|\nabla f(y)\|_{\ast}
\quad\text{and}\quad\|P^{\top}\nabla f(y)\|_{\ast}>\epsilon,
\end{align*}
where $\kappa\in(0,\min(1,\|P\|_{\ast}))$ and $\epsilon\in(0,1)$. 
These conditions prevent reliance on the coarse model when the direction  $(x_{k}^+-Ry_{k})$
is nearly zero.

\vspace{1mm}
\noindent
\textbf{Box constrained coarse model.}
Given $x_k := Ry_k$, the coarse model \eqref{eq:coarse-model-1} can be extended to the box-constrained 
problem $\min_{y \in [l, u]} f(y)$ as follows
\begin{equation}
\min_{x\in \mathbb{R}^N}\psi_k(x)\qquad\text{subject to}\quad l_{k,P}\leq x\leq u_{k,P},
\label{eq:box-coarse-model}
\end{equation}
where the bounds are redefined as in \cite{gratton_recursive_2008}
\begin{align}
\left(l_{k,P}\right)_{j} & =x_{k,j}+\frac{1}{\|P\|_{\infty}}\max_{i=1,\dots n}\begin{cases}
(l-y_{k})_{i}, & P_{ij}>0,\\
(y_{k}-u)_{i}, & P_{ij}<0,
\end{cases}\label{eq:coarse-box-lower}\\
\left(u_{k,P}\right)_{j} & =x_{k,j}+\frac{1}{\|P\|_{\infty}}\min_{i=1,\dots,n}\begin{cases}
(u-y_{k})_{i}, & P_{ij}>0,\\
(y_{k}-l)_{i}, & P_{ij}<0,
\end{cases}\label{eq:coarse-box-upper}
\end{align}
and accommodates negative elements in $P$, as in cubic interpolation. The next result ensures that 
feasibility is preserved under prolongation.

\begin{lemma}[{\cite[Lemma 4.3]{gratton_recursive_2008}}]
\label{lem:restricted-box}
Let $l\leq y_k\leq u$ and $x_k=Ry_k$. Let $P:\mathbb{R}^{N}\rightarrow\mathbb{R}^{n}$ and
$l_{P,k}$ and $u_{P,k}$ from \eqref{eq:coarse-box-lower}, \eqref{eq:coarse-box-upper}. Then 
\begin{align*}
	l\leq y_k+P(x_k-x)\leq u,\qquad \forall x\in[l_{P,k},u_{P,k}].
\end{align*}
\end{lemma}
To define a suitable coarse correction condition, we follow \cite[Note 2]{gratton_recursive_2008} and set
\begin{align*}
a_{k} & :=\left\Vert \Pi_{f}[y_{k}-\nabla f(y_{k})]-y_{k}\right\Vert _{\ast},
\qquad b_{k}:=\left\Vert \Pi_{\psi_{k}}[x_{k}-\nabla\psi(x_{k})]-x_{k}\right\Vert _{\ast},
\end{align*}
where $\Pi_{f}$ and $\Pi_{\psi_k}$ are Euclidean projection on the
constraint sets of $f$ and $\psi_{k}$ (\ref{eq:coarse-box-lower}, \ref{eq:coarse-box-upper}), respectively. 
A coarse correction step is taken if
$b_{k}\geq\kappa a_{k}$ 
for some constant $\kappa\in(0,1)$. 

The step size $\alpha_{k}$
in the update \eqref{eq:coarse-update} can be adapted to constraints using a projected line search, such as the Armijo rule along the projection arc \cite[pp. 227--230]{bertsekasNonlinearProgramming1995}. However, by Lemma \ref{lem:restricted-box}, we can directly set 
$\alpha_{k}=1$.

\section{Rate of Convergence}
\label{sec:Rate-of-Convergence}
In this section, we show that Algorithm \ref{alg:Two-level-algorithm} achieves a linear convergence rate and requires only finitely many coarse correction steps, under the assumption of gradient Lipschitz continuity and strong convexity of the objectives.

\vspace{1mm}
\noindent
\textbf{Sufficient decrease.}
A key requirement is that the step size generated by the \emph{Armijo condition},
\begin{align}
f(y_{k}+\alpha_{k}d_{k}) & \leq f(y_{k})+\rho_{1}\alpha_{k}\left[\nabla f(y_{k})\right]^{\top}d_{k},\quad\rho_{1}\in(0,0.5),\label{eq:armijo-condition}
\end{align}
must be bounded from below.

\begin{lemma}[{\cite[Theorem 3.3]{wen_line_2010}}]\label{lem:armijo-descent}
Suppose $\nabla f$ is Lipschitz continuous with constant $M_f$. Let $\rho_{1}\in(0,0.5)$ and $d_{k}$ a descent direction at $y_{k}$. Then step sizes 
\begin{align*}
\alpha\in[0,\hat{\alpha}], \quad\hat{\alpha} = \frac{2(\rho_{1}-1)d_{k}^{\top}\nabla f(y_{k})}{M_f\|d_{k}\|^{2}}.
\end{align*}
satisfy (\ref{eq:armijo-condition}). Furthermore, backtracking line search terminates with
\begin{align*}
\min\left(\alpha_{p},\frac{2\beta(\rho_{1}-1)d_{k}^{\top}\nabla f(y_{k})}{M_f\|d_{k}\|^{2}}\right) & \leq\alpha_{k}\leq\alpha_{\rho}
\end{align*}
where $\beta$ is the step size reduction parameter and $\alpha_{\rho}$
the initial step size.
\end{lemma}

\vspace{1mm}
\noindent
\textbf{Coarse correction step.}
Next, we provide a lower bound for the decrease in the function value at the fine grid after a coarse correction step (Algorithm \ref{alg:Two-level-algorithm}, Line \ref{f-lesser-value}), under the following two assumptions.

\begin{assumption}\label{asu:hessian-bounded-eigenvalues}
Let $y_0$ denote the initial iterate. On the level sets $\mathcal{L}_{f}(y_{0})$ and $\mathcal{L}_{g}({Ry_{0}})$, the fine and coarse objectives have Lipschitz continuous gradient with constants $M_f$ and $M_g$, respectively, and are strongly convex with constants
 $m_f$ and $m_g$. 
\end{assumption}

\begin{assumption}\label{asu:coarse-exact-solution}
The coarse model (\ref{eq:coarse-model-1}) is solved exactly. In Algorithm \ref{alg:Two-level-algorithm},
replace Line \ref{mls-find-lesser-value} with ``Find $x_k^\ast = \argmin{\psi_k}$.''
\end{assumption}

\begin{lemma}[Sufficient decrease by coarse correction]\label{lem:coarse-decrease}
Suppose Assumption \ref{asu:hessian-bounded-eigenvalues} and \ref{asu:coarse-exact-solution} hold.
The coarse correction step ${d}_{k}$ in Algorithm \ref{alg:Two-level-algorithm} then
leads to a reduction in function value
\begin{align*}
f(y_{k})-f(y_{k}+\alpha_{k}{d}_{k}) & \geq\rho_{1}\beta\frac{m_{g}^{3}\kappa^{2}}{M_{f}M_{g}^{2}\omega^{2}}\|\nabla f(y_{k})\|_{\ast}^{2},
\end{align*}
where $\rho_{1},\kappa$ and $\beta$ are defined
in Algorithm \ref{alg:Two-level-algorithm} and 
$\omega:=\max\left\{ \|R\|,\|P\|\right\} $.
\end{lemma}

\begin{remark}
Compared to \cite[Lemma 3.1]{ho_newton-type_2019}, we observe cubic terms \( m_{g}^3 \) and quadratic terms \( M_{g}^2 \) in the reduction lower-bound for a coarse correction step.
However, these constants pertain to the \emph{coarse} model and may differ significantly -- by several orders of magnitude -- from those of the \emph{fine} model.
\end{remark}
\begin{remark}
The parameter $\kappa\in(0,1)$ dictates when the algorithm
performs a coarse correction step. Choosing $\kappa$ too small is
not desirable in terms of worst-case complexity. If $\kappa$ is chosen
too large, coarse correction steps are less likely.
\end{remark}

\vspace{1mm}
\noindent
\textbf{Linear rate of convergence.}
Lemma \ref{lem:coarse-decrease} on the coarse correction step, together with the following assumption from \cite{ho_newton-type_2019} on the fine correction step (Line \ref{mls-fine-correction} in Algorithm \ref{alg:Two-level-algorithm}), will ensure sufficient decrease and a linear rate of convergence.

\begin{assumption}\label{asu:fine-descent-condition}
For fine correction steps there exists a constant $\Lambda_{f}>0$ such that 
\begin{equation}
f(y_{k})-f(y_{k+1})\geq\Lambda_{f}\|\nabla f(y_{k})\|_{\ast}^{2}.
\label{eq:fine-decrease-bound}
\end{equation}
\end{assumption}
This holds for gradient descent, (inexact) Newton, and L-BFGS \cite{wen_line_2010}.
\begin{theorem}[Sufficient decrease for two level optimization]\label{thm:mls-decrease}
Suppose Assumption \ref{asu:hessian-bounded-eigenvalues}, \ref{asu:coarse-exact-solution} and \ref{asu:fine-descent-condition} hold.
The step $d_k$ in Algorithm \ref{alg:Two-level-algorithm} leads to a decrease in function value
\begin{align*}
f(y_k)-f(y_{k+1})\geq & \min\left\{ \Lambda_{f},\rho_{1}\beta\frac{m_{g}^{3}\kappa^{2}}{M_{f}M_{g}^{2}\omega^{2}}\right\}\|\nabla f(y_k)\|_{\ast}^{2},
\end{align*}
where the constant $\Lambda_{f}$ depends on the fine correction step chosen.
\end{theorem}
Now the linear rate of convergence follows immediately.
\begin{corollary}\label{cor:mls-linear-rate}
Let $\left( y_{k}\right) _{k\geq0}$ be the sequence generated by
Algorithm \ref{alg:Two-level-algorithm} and $y_{\ast} = \argmin f$. Assume $M_{g}^{2}\geq m_{g}^{3}$ which holds for e.g. $M_{g}>1$, $m_{g}<1$. Then, 
\begin{align*}
f(y_k)-f(y_\ast) & \leq(1-2m_{f}\Lambda)^{k}[f(y_0)-f(y_\ast)].
\end{align*}
\end{corollary}
We denote the constant in Theorem \ref{thm:mls-decrease} by $\Lambda$. As shown below, Algorithm \ref{alg:Two-level-algorithm} takes a finite number of coarse correction steps when $f$ is coercive. We then have $\Lambda_f \equiv \Lambda$ for $k$ sufficiently large.


\vspace{1mm}
\noindent
\textbf{Number of coarse correction steps.} We show that Algorithm \ref{alg:Two-level-algorithm} only performs a finite number of coarse correction steps, depending on the distance of $y_0$ to the exact solution $y_*$. We define the set diameter
\begin{align*}
	\mathcal{R}(y_{0}) = \max_{y\in\mc{L}_f(y_0)} {\|y-y_{\ast}\| }
\end{align*}
and establish a sublinear convergence rate dependent on it.

\begin{theorem}[{\cite[Theorem 3.5]{ho_newton-type_2019}} Sublinear convergence rate]
Let $\left ( y_{k}\right ) _{k\geq0}$ be the sequence generated by
Algorithm \ref{alg:Two-level-algorithm}. Then
\begin{align*}
f(y_k)-f(y_\ast) & \leq\frac{\mathcal{R}^{2}(y_{0})}{\Lambda}\frac{1}{2+k}.
\end{align*}
\end{theorem}
Using this result and a criterion for fine correction steps, we determine an upper bound on the number of coarse correction steps. Specifically, if $f$ is coercive and $\mathcal{R}(y_0)$
 is bounded, Algorithm \ref{alg:Two-level-algorithm} performs only finitely many coarse corrections. The convergence rate then matches that of the fine correction method for sufficiently large 
$k$. However, numerical experiments show that incorporating coarse corrections significantly improves performance compared to using fine corrections alone; see Section \ref{sec:Tomography} for details.
\begin{lemma}[{\cite[Lemma 3.6]{ho_newton-type_2019}}]
No coarse correction step will be taken by Algorithm \ref{alg:Two-level-algorithm}
when $\|\nabla f(y_{k})\|\leq\frac{\epsilon}{\omega}$, where $\omega=\max\left\{ \|P\|,\|R\|\right\} $
and $\epsilon>0$ is user-defined.
\end{lemma}
\begin{lemma}[{\cite[Lemma 3.7]{ho_newton-type_2019}}]
No coarse correction step will be taken after
\[
\left(\frac{\omega}{\epsilon}\right)^{2}\frac{\mathcal{R}^{2}(y_{0})}{\Lambda^{2}}-2\qquad\text{iterations.}
\]
\end{lemma}
\begin{remark}
The smaller the choice of $\epsilon>0$, the more coarse correction
steps will be taken. Depending on the descent method for the fine correction steps, the choice of $\epsilon$ could be different.
For example, a good choice for Newton's method with $\|\cdot\|_2$ and $d_{k}=-\left[\nabla^{2}f(y_{k})\right]^{-1}\nabla f(y_{k})$
is $\epsilon=3\omega(1-2\rho_{1})m_{f}^{2}/M_{f}$
\cite[p.15]{ho_newton-type_2019}.
\end{remark}
\begin{remark}
In Section \ref{sec:Multilevel-Optimization}, we discussed the box-constrained coarse model (\ref{eq:coarse-box-lower}, \ref{eq:coarse-box-upper}), which depends on the current fine-level iterate $y_k$. For functions where the Hessian's condition number varies across the domain (e.g., that involve the Bregman divergence), this poses challenges for establishing convergence rates unless the coarse constraints can be uniformly bounded.
\end{remark}

An alternative approach is to approximate the active set of constraints using projected gradient descent and apply the unconstrained multilevel method to the inactive set. The step size $\alpha_{k}$ 
 should be adjusted to ensure feasibility of the iterates. If the active set changes, gradient projection iterations can be used to update it as needed \cite{bertsekasNonlinearProgramming1995}. Exploring such extensions is left for future work.

\section{Experiments}
\label{sec:Tomography}

We compare Algorithm \ref{alg:Two-level-algorithm} to basic (projected) gradient descent and the state-of-the-art quasi-Newton method L-BFGS \cite{byrd_limited_1995,zhu_algorithm_1997}, which handles box constraints. The comparison includes the proposed coarse models \eqref{eq:coarse-model-1} and \eqref{eq:box-coarse-model}, as well as the algebraic coarse model \eqref{eq:algebraic-coarse}. Tests are conducted on both \emph{unconstrained} and \emph{box-constrained} variational models for discrete tomography. Whereas the relevant objectives are not strongly convex, Algorithm \ref{alg:Two-level-algorithm} exhibits the same behavior as predicted for strongly convex functions in Section \ref{sec:Rate-of-Convergence}.

\vspace{1mm}
\noindent
\textbf{Data setup.} We reconstruct a large scale image $y$ ($n=1024\times 1024$) subsampled using a tomographic projection matrix
 $A\in\mathbb{R}^{m\times n}$, 
generated with ASTRA-toolbox. 
 Parallel beam projections are taken at equidistant angles between
$0$ and $\pi$, with a 10\% undersampling rate at the fine grid. 


\vspace{1mm}
\noindent
\textbf{Unconstrained optimization.}
We first consider the unconstrained problem  
\begin{align}
\min_{y \in \mathbb{R}^{n}} f(y) := \frac{1}{2}\|Ay - b\|_{2}^{2} + \lambda L_{\rho}(Dy), \quad \lambda = 0.1, \quad \rho = 0.01,  
\label{eq:tomography-unconstrained}
\end{align}  
where \( L_{\rho} \) is the Pseudo-Huber function defined as \( L_\rho(a) = \sum_{i=1}^n (\sqrt{\rho^2 + a_i^2} - \rho) \), and \( D \) is the discretized image gradient using forward differences.  
The objective \( f \) has a Lipschitz continuous gradient but lacks strong convexity without addition of a $\mu\|\cdot\|_2$ regularizer (which offered no improvement in our experiments.)

\vspace{1mm}
\noindent
\textbf{Lipschitz constants.}
We analyze the grid-dependent smoothness of the geometric coarse model \eqref{eq:coarse-model-1} compared to the algebraic coarse model \eqref{eq:algebraic-coarse}. Ignoring the linear term in \eqref{eq:coarse-model-1}, the Lipschitz constant of $\psi_k $ equals that of the objective $f$ \eqref{eq:tomography-unconstrained} discretized on the respective grid. Similarly, ignoring the linear terms in $\varphi_k $ \eqref{eq:algebraic-coarse}, the quadratic term  
\begin{align*}
\frac{1}{2} \left\langle x - Ry_k, Q_k(x - Ry_k) \right\rangle, \quad Q_k = R \nabla^2 f(y_k) P  
\end{align*}  
yields the Lipschitz constants for the gradients of $\psi_k$ and $ \varphi_k $  
\begin{equation}
L_{\nabla \psi_k} = \|A\| \|A^\top\| + \frac{8\lambda}{\rho}, \quad  
L_{\nabla \varphi_k} = \|Q_k\|  
\leq \|R\| \|\nabla^2 f(y_k)\| \|P\| \leq \omega^2 L_{\nabla f},  
\end{equation}  
where $ \omega = \max(\|R\|, \|P\|) $. For concrete numerical values, see Table \eqref{tab:lipschitz-constants}.

\begin{table}[hbt]
  \begin{tabular}{c|c|cc|}
    Grid size & $L_{\nabla \psi_k}$ & $L_{\nabla \varphi_k}$ \\
    \hline
    32 & $4.5480\cdot 10^3$ & $7.276793\cdot 10^4$ & 	\\
    64 & $9.4144\cdot 10^3$ & $1.506304\cdot 10^5$ & 	\\
    128 & $1.9176\cdot 10^4$ & $3.068164\cdot 10^5$ & 	\\
  \end{tabular}
  \hfill
  \begin{tabular}{c|c|cc|}
    Grid size & $L_{\nabla \psi_k}$ & $L_{\nabla \varphi_k}$ \\
    \hline
    256 & $3.8284\cdot 10^4$ & $6.125398\cdot 10^5$ & 	\\
    512 & $7.6431\cdot 10^4$ & $1.222901\cdot 10^6$ 	& 	\\
    1024 & $1.5297\cdot 10^5$ & $2.447472\cdot 10^6$ &
  \end{tabular}
  \caption{Lipschitz constants for geometric \eqref{eq:coarse-model-1} and algebraic \eqref{eq:algebraic-coarse} coarse models for  $\rho=0.01$, $\lambda=0.1$.}
  \label{tab:lipschitz-constants}
\end{table}

\vspace{1mm}
\noindent
\textbf{Multilevel implementation.}
We use the recursive version of Algorithm \ref{alg:Two-level-algorithm} applied 
to the Bone phantom  of size $n=1024\times 1024$, compare Fig. \ref{fig:bone-reconstruction}, with a total of 6 levels (with coarsest grid size $32\times32$).  At the finest level, we perform $k=200$ outer iterations and $k=1$ at coarser levels $\ell=1,\dots,8$. Coarse correction steps use the constants $\kappa=0.47$ and $\epsilon=10^{-3}$. For the information transfer between levels we set $R=\frac{1}{16}I_{\ell}^{\ell-1}$ and $P=4R^{T}$, where $I_{\ell}^{\ell-1}$ is the full weighting operator. 
Fine correction steps are done with gradient descent with $2^{\ell}$ iterations for $\ell\in\{0,\dots,5\}$. \; Step sizes for coarse correction steps and gradient descent are determined
through Wolfe line search, using $c_{1}=10^{-4}$ and $c_{2}=0.9$. \; 
Results are compared to gradient descent with Wolfe line search on the finest level, and L-BFGS with 3 correction pairs.
We summarize our results in
Figure \ref{fig:results} (top row).


\vspace{1mm}
\noindent
\textbf{Constrained optimization.}
We consider the bound constrained problem
\begin{align}
f(y)=\min_{y\in\mathbb{R}_{+}^{n}}\text{KL}(Ay,b)=\min_{y\in\mathbb{R}_{+}^{n}}\left\langle Ay,\log Ay / b\right\rangle -\left\langle y-b,\mathbbm{1}\right\rangle ,
\label{eq:tomography-constrained}
\end{align}
which has no Lipschitz-continuous gradient on the domain $\mathbb{R}^n_+$. To control behavior of the gradient 
on the boundary, we consider the \emph{thresholded} domain 
$\mathbb{R}^n_{\geq\beta}:=\{y\in\mathbb{R}^n: y_i\geq\beta>0,i\in[n]\}$. The parameters for multilevel optimization are 
as in the unconstrained case, except we use projected gradient descent in each level. \;
Step sizes for coarse correction steps and gradient descent are determined with Armijo line search along 
the projection arc, using $c_{1}=10^{-4}$ and $\beta=0.8$. \;
Results are compared to gradient descent with line search along the projection arc on the finest level, 
and L-BFGS-B with 3 correction pairs. We summarize our results in
Figure \ref{fig:results} (bottom row).

\section{Conclusion}
\label{sec:Conclusion}

In this work, we established precise convergence results for an unconstrained multilevel 
method using Nash's classical geometric coarse model \cite{nash_multigrid_2000}. The derived 
convergence rate applies without explicitly incorporating the Hessian of the fine-grid 
objective as done in \cite{ho_newton-type_2019}, which our approach avoids. We extended 
the method to box-constrained problems and applied it to a discrete tomography problem. 
Despite severe undersampling, the method converges rapidly when far from the solution, 
outperforming state-of-the-art techniques. However, near the solution, methods with 
superlinear convergence, such as L-BFGS, prove more effective.

\vspace{2mm}
\textbf{Acknowledgments.} This work is funded by Deutsche Forschungsgemeinschaft (DFG) under 
Germany’s Excellence Strategy EXC-2181/1 - 390900948 (the Heidelberg STRUCTURES Excellence 
Cluster).

\textbf{Disclosure of Interests.} The authors have no competing interests to declare that are 
relevant to the content of this article.

\begin{figure}[ht]
  \hspace*{-1.2cm}\begin{tabular}{cc}
    \includegraphics[width=0.55\columnwidth]{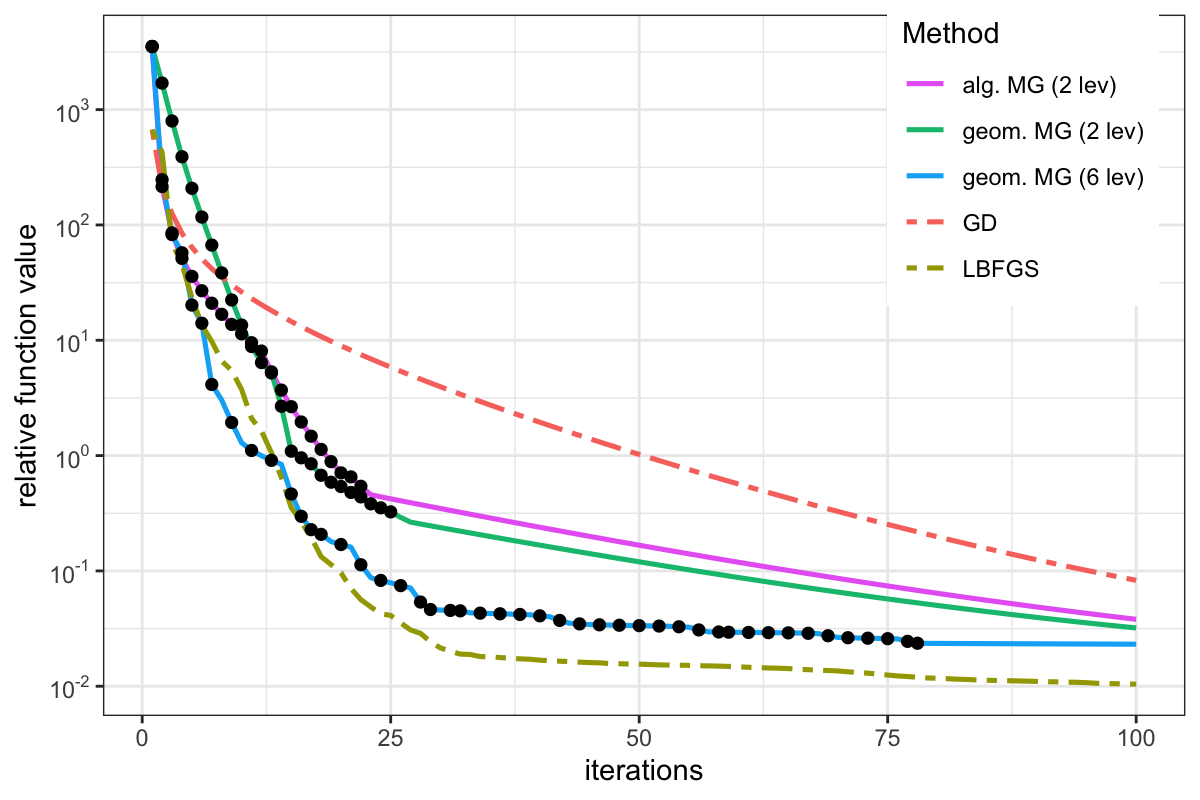}
  & \includegraphics[width=0.55\columnwidth]{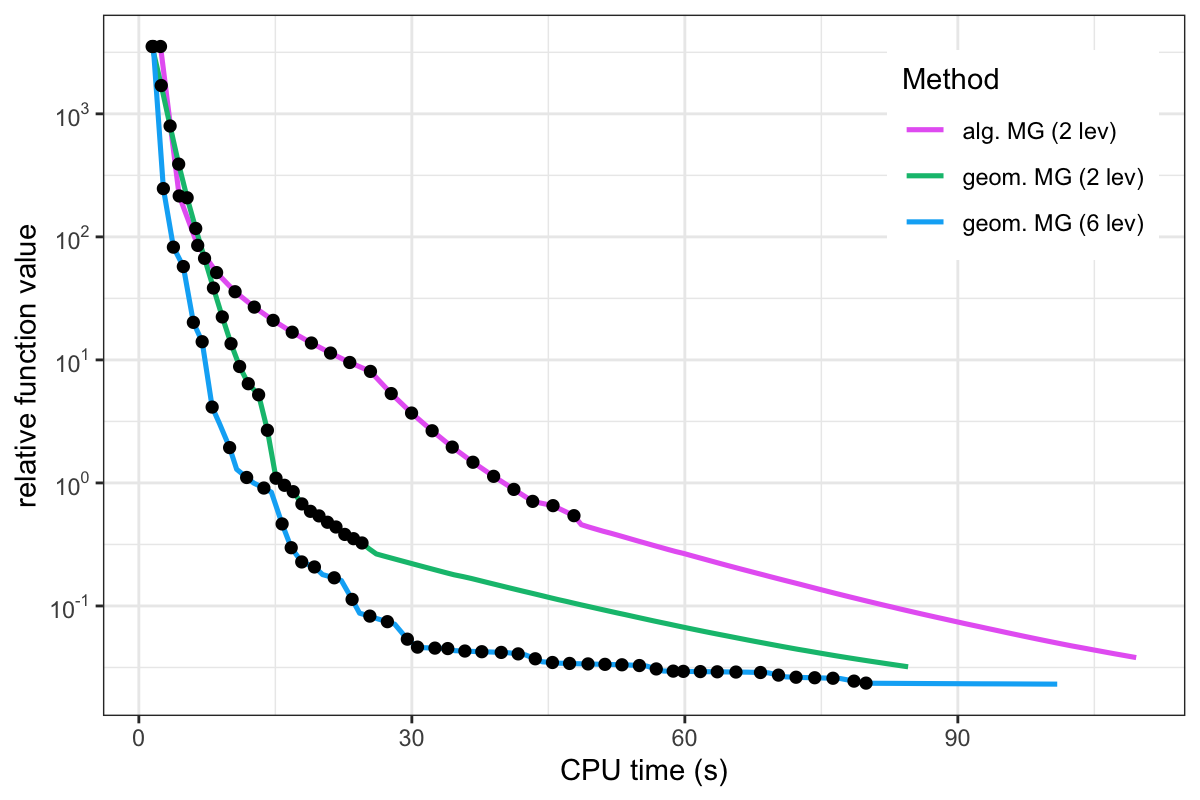} \\
    \includegraphics[width=0.55\columnwidth]{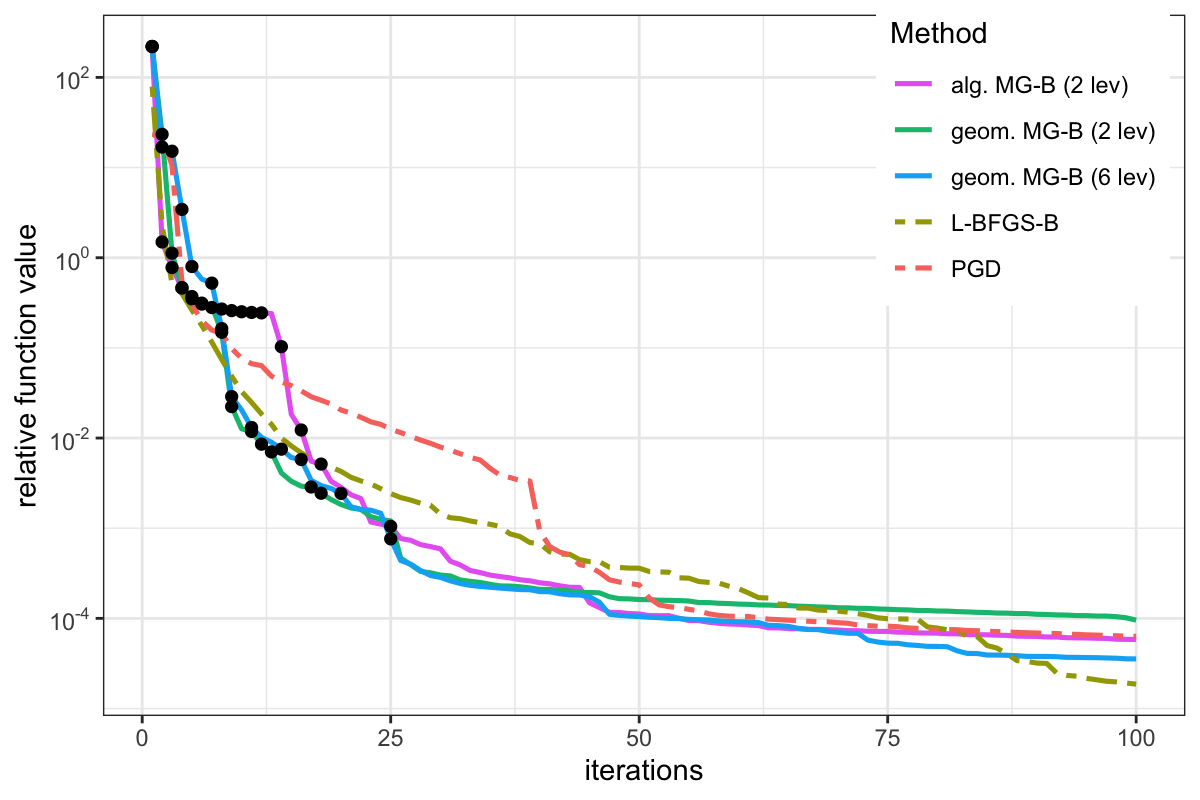}
  & \includegraphics[width=0.55\columnwidth]{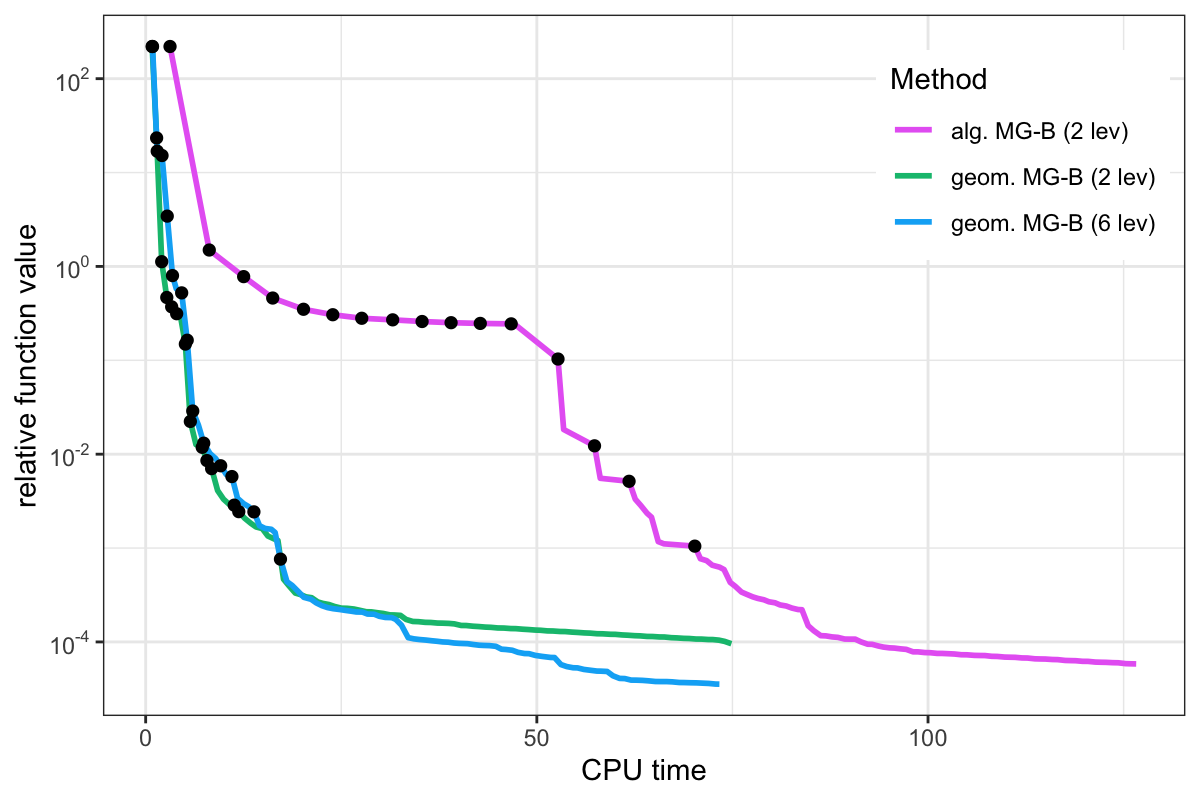} \\
  \end{tabular}
  \caption{\textbf{Comparison of algorithms.} The \emph{relative function value} (left) is 
  plotted against iteration count and \emph{CPU time} (right) for the unconstrained (top) 
  and constrained (bottom) discrete tomography problems in Section \ref{sec:Tomography}. 
  CPU time for L-BFGS is not included, as the C implementation is not directly comparable 
  to the multigrid MATLAB implementation. Adding more levels improves performance. 
  Furthermore, the \emph{geometric coarse models} \eqref{eq:coarse-model-1} and 
  \eqref{eq:box-coarse-model} outperform the \emph{algebraic coarse model}
   \eqref{eq:algebraic-coarse}. In the initial phase, the multilevel approach reduces 
   the objective more rapidly than L-BFGS.}
  \label{fig:results}
  \end{figure}

\bibliographystyle{plain}
\bibliography{bibliography}

\section*{Appendix}
We use the following \emph{growth property }for strongly convex functions.
\begin{lemma}[{\cite[Corollary 2.2.1]{nesterov_lectures_2018}}]
\label{lem:growth-property}
Let $f\in\mathscr{S}_{\mu}^{1}(\mathbb{R}^{n})$ and $\nabla f(y^{\ast})=0$. 
Then $f(y)\geq f(y^{\ast})+\frac{1}{2}\mu\|y-y^{\ast}\|^{2}$
for all $y\in\mathbb{R}^{n}$.
\end{lemma}

\subsection*{Proof of Lemma \ref{lem:coarse-decrease}}
Let $x_k := Ry_k$ and $x_\ast$ a solution to (\ref{eq:coarse-model-1}). 
By Lemma \ref{lem:coarse-descent-direction} and Lemma \ref{lem:growth-property},
\begin{align*}
\left\langle \nabla f(y_k),d_k\right\rangle  & \leq\psi_k(x_\ast)-\psi_k(x_k)\leq-\frac{m_g}{2}\|x_{\ast}-x_{k}\|^{2}
\end{align*}
By Lipschitz continuity of $\nabla\psi_k$,
\begin{align*}
\|x_{k}-x_{\ast}\| & \geq\frac{1}{M_g}\|\nabla\psi_k(x_k)-\nabla\psi_k(x_\ast)\|_{\ast}=\frac{1}{M_g}\|\nabla\psi_k(x_k)\|_{\ast}
\end{align*}
and therefore
\begin{align}
\left\langle \nabla f(y_k),d_k\right\rangle  & \leq-\frac{m_g}{2M_g^{2}}\|\nabla\psi_k(x_k)\|_{\ast}^{2}\label{eq:lemma-descent-1}
\end{align}
Let $\tilde{d}_{k}\equiv x_\ast-x_k$. We have 
\[
\|d_k\|^{2}=\|P\tilde{d}_{k}\|^{2}\leq\|P\|^{2}\|\tilde{d}_{k}\|^{2}\leq\omega^{2}\|x_\ast-x_k\|^{2}
\]
By strong convexity of $\psi_k$,
\begin{align*}
\|x_\ast-x_k\|^{2} & \leq\frac{1}{m_g^{2}}\|\nabla\psi_k(x_k)\|_{\ast}^{2}
\end{align*}
by \cite[Theorem 2.1.10]{nesterov_lectures_2018} and therefore
\begin{equation}
\|d_k\|^{2}\leq\frac{\omega^{2}}{m_g^{2}}\|\nabla\psi_k(x_k)\|_{\ast}^{2}.\label{eq:lemma-descent-2}
\end{equation}
By first-order coherency (\ref{eq:first-order-coherence}) and the
condition for a coarse correction step,
\begin{equation}
\|\nabla\psi_k(x_k)\|_{\ast}^{2}=\|R\nabla f(y_k)\|_{\ast}^{2}\geq\kappa^{2}\|\nabla f(y_k)\|_{\ast}^{2},
\quad \|\nabla \psi_k(x_k)\|_{\ast} > 0.
\label{eq:lemma-descent-3}
\end{equation}
Let $c_{k}:=\left\langle d_k,\nabla f(y_k)\right\rangle <0$.

By Lemma \ref{lem:armijo-descent} the Armijo line search terminates with 
\begin{align*}
\alpha_{k} & \geq\frac{2\beta(\rho_{1}-1)}{M_f\|d_k\|^{2}}c_{k} \overset{(\ref{eq:lemma-descent-2})}{\geq}
\frac{2\beta(\rho_{1}-1)}{\|\nabla\psi_k(x_k)\|_{\ast}^{2}}\frac{m_g^{2}}{M_f\omega^{2}}c_{k}\\
\left(-c_{k}>c_{k}(\rho_{1}-1)>-\frac{1}{2}c_{k}\right)\qquad 
& >\frac{-\beta}{\|\nabla\psi_k(x_k)\|_{\ast}^{2}}\frac{m_g^{2}}{M_f\omega^{2}}c_{k}
\end{align*}
Therefore by the Armijo condition,
\begin{align*}
f(y_k+\alpha_{k}d_k)-f(y_k) \leq\rho_{1}\alpha_{k}c_{k}
& \,\,\leq-\rho_{1}\beta\frac{1}{\|\nabla\psi_k(x_k)\|_{\ast}^{2}}\frac{m_g^{2}}{M_f\omega^{2}}c_{k}^{2}\\ 
& \overset{(\ref{eq:lemma-descent-1})}{\leq}-\rho_{1}\beta
\frac{m_g^{3}}{M_fM_g^{2}\omega^{2}}\|\nabla\psi_k(x_k)\|_{\ast}^{2}\\
& \overset{(\ref{eq:lemma-descent-3})}{\leq}-\rho_{1}\beta
\frac{m_g^{3}\kappa^{2}}{M_fM_g^{2}\omega^{2}}\|\nabla f(y_k)\|_{\ast}^{2}
\end{align*}

\subsection*{Proof of Corollary \ref{cor:mls-linear-rate}}

By strong convexity of $f$, we have
\begin{align*}
f(y_k)-f(y_\ast) & \leq\frac{1}{2m_f}\|\nabla f(y_k)\|^{2}.
\end{align*}
Combining this with Theorem \ref{thm:mls-decrease} yields for $k\geq1$
\begin{align*}
[f(y_{k-1})-f(y_\ast)]-[f(y_k)-f(y_\ast)] & =f(y_{k-1})-f(y_k) \geq\Lambda\|\nabla f(y_{k-1})\|_{2}^{2}\\
 & \geq2m_f\Lambda[f(y_{k-1})-f(y_\ast)]
\end{align*}
or equivalently, 
\begin{align*}
f(y_k)-f(y_\ast) & \leq(1-2m_f\Lambda)[f(y_{k-1})-f(y_\ast)]\\
 & \leq\cdots\leq(1-2m_f\Lambda)^{k}[f(y_0)-f(y_\ast)].
\end{align*}

\end{document}